\documentclass[11pt]{amsart}
\usepackage{amssymb,latexsym,amsmath}
\begin{document}

\begin{abstract} 
We prove an algebraicity result for the central critical value of certain Rankin-Selberg $L$-functions
for ${\rm GL}_n \times {\rm GL}_{n-1}$. This is a generalization and refinement of the results of 
Harder \cite{harder}, Kazhdan, Mazur and Schmidt \cite{kazhdan-mazur-schmidt}, Mahnkopf \cite{mahnkopf-jussieu}, and 
Kasten and Schmidt \cite{kasten-schmidt}.  
As an application of this result, we prove algebraicity results for certain critical values of the fifth and the seventh symmetric power $L$-functions attached to a holomorphic
cusp form. Assuming Langlands' functoriality one can prove similar algebraicity results for the special values
of any odd symmetric power $L$-function. We also prove a conjecture of Blasius and Panchishkin on twisted 
$L$-values in some cases. 
We comment on the compatibility of our results with Deligne's conjecture on the critical values of 
motivic $L$-functions. These results, as in the above mentioned works, are, in general, based on a nonvanishing hypothesis on certain archimedean integrals. 
 \end{abstract}

\title[]{On the special values of certain Rankin--Selberg $L$-functions and  
applications to odd symmetric power $L$-functions of modular forms}

\author{\bf A. Raghuram}
\date{\today}
\subjclass[2000]{11F67 (11F70, 11F75, 22E55)}

\maketitle

\tableofcontents

\numberwithin{equation}{section}
\newtheorem{thm}[equation]{Theorem}
\newtheorem{cor}[equation]{Corollary}
\newtheorem{lemma}[equation]{Lemma}
\newtheorem{prop}[equation]{Proposition}
\newtheorem{con}[equation]{Conjecture}
\newtheorem{ass}[equation]{Assumption}
\newtheorem{defn}[equation]{Definition}
\newtheorem{rem}[equation]{Remark}
\newtheorem{exer}[equation]{Exercise}
\newtheorem{exam}[equation]{Example}
\newtheorem{hypo}[equation]{Hypothesis}

\section{Introduction and statements of theorems}
\label{sec:intro}

Let $\Pi$ (respectively, $\Sigma$) be a regular algebraic cuspidal automorphic representation of 
${\rm GL}_n({\mathbb A})$ (respectively, ${\rm GL}_{n-1}({\mathbb A})$); here 
${\mathbb A}$ is the ad\`ele ring of ${\mathbb Q}$. We assume the representations are such that 
$s = 1/2$ is critical for the Rankin--Selberg $L$-function attached to $\Pi \times \Sigma$. We prove 
an algebraicity result for $L(1/2, \Pi \times \Sigma)$. See Theorem~\ref{thm:rankin-selberg}.
This is a generalization and refinement of the results of 
Harder \cite{harder}, Kazhdan, Mazur and Schmidt \cite{kazhdan-mazur-schmidt}, Mahnkopf \cite{mahnkopf-jussieu}, and 
Kasten and Schmidt \cite{kasten-schmidt}. Our result, as in the above mentioned works, is, in general, based on a nonvanishing hypothesis on certain archimedean integrals. We also prove a conjecture of Blasius and Panchishkin on twisted $L$-values in some cases using the period relations proved in our paper with Shahidi \cite{raghuram-shahidi-imrn}; see Theorem~\ref{thm:twisted}.

Let $\varphi$ be a holomorphic cusp form of weight $k$. We consider twisted odd symmetric power $L$-functions
$L(s, {\rm Sym}^{2n-1}\varphi, \xi)$, where $\xi$ is any Dirichlet character. Using the above result on Rankin--Selberg $L$-functions, we prove algebraicity results for certain critical values of $L(s, {\rm Sym}^{2n-1}\varphi, \xi)$ when $n \leq 4$. See Theorem~\ref{thm:sym-357}.
For $n=1$ this is a classical theorem due to Shimura \cite{shimura2}; indeed, in this case, our theorem boils
down to Harder's proof \cite{harder} of Shimura's theorem. For $n=2$, our proof may be regarded as a new proof of 
the result of Garrett and Harris \cite{garrett-harris} on symmetric cube $L$-functions. Our theorem is new for 
the fifth and seventh symmetric power $L$-functions. 
Assuming Langlands' functoriality one can prove similar algebraicity results for any odd symmetric power $L$-function.

We now describe the theorems proved in this paper in greater detail, toward which we need some notation. 
Given a regular algebraic cuspidal automorphic representation $\Pi$ of 
${\rm GL}_n({\mathbb A})$ one knows (from Clozel \cite{clozel}) that there is a pure dominant integral weight 
$\mu$ such that $\Pi$ has a nontrivial contribution to the cohomology of some locally symmetric space of ${\rm GL}_n$  
with coefficients coming from the dual of the finite dimensional representation with highest weight $\mu$. 
 We denote this 
as $\Pi \in {\rm Coh}(G_n, \mu^{\vee})$, for $\mu \in X_0^+(T_n)$, where $T_n$ is the diagonal torus of 
$G_n = {\rm GL}_n$.
Under this assumption on $\Pi$, one knows that its rationality field
${\mathbb Q}(\Pi)$ is a number field, and that $\Pi$ is defined over this number field. It is further known that 
the Whittaker model of $\Pi$ carries a ${\mathbb Q}(\Pi)$-structure, and likewise, a suitable cohomology space 
also carries a rational structure. One defines a period $p^{\epsilon}(\Pi)$ by comparing these rational structures; 
here $\epsilon$ is a sign which can be arbitrary if $n$ is even, and is uniquely determined by $\Pi$ if $n$ is odd.  
We briefly review the definition of these periods in \ref{subsec:periods}, and refer the reader to 
\cite{raghuram-shahidi-imrn} for more details. We now state one of the main theorems of this paper:

\begin{thm}
\label{thm:rankin-selberg}
Let $\Pi$ (resp., $\Sigma$) be a regular algebraic cuspidal automorphic representation of ${\rm GL}_n({\mathbb A})$ (resp., ${\rm GL}_{n-1}({\mathbb A})$).  
Let $\mu \in X^+_0(T_n)$ be such that $\Pi \in {\rm Coh}(G_n,\mu^{\vee})$, and let $\lambda \in X^+_0(T_{n-1})$ be such that $\Sigma \in {\rm Coh}(G_{n-1},\lambda^{\vee})$. Assume that
$\mu^{\vee} \succ \lambda$ (see \S\ref{sec:prelims} for the definition and a consequence of this condition). 
Assume also that $s=1/2$ is critical for 
$L_f(s, \Pi \times \Sigma)$ which is the finite part of the Rankin--Selberg $L$-function attached to the 
pair $(\Pi,\Sigma)$. There exists canonical signs
$\epsilon, \eta \in \{\pm\}$ attached to the pair $(\Pi,\Sigma)$; there exists 
nonzero complex numbers $p^{\epsilon}(\Pi)$, $p^{\eta}(\Sigma)$, and
assuming the validity of Hypothesis~\ref{hypo:nonvanishing} there exists a nonzero complex number  
$p_{\infty}(\mu,\lambda)$, such that for any $\sigma \in {\rm Aut}({\mathbb C})$ we have 
$$
\sigma\left(\frac{L_f(1/2,\Pi \times \Sigma)}
{p^{\epsilon}\, (\Pi)p^{\eta}\, (\Sigma)\mathcal{G}(\omega_{\Sigma_f})\, p_{\infty}(\mu,\lambda)}\right)
\  = \  
\frac{L_f(1/2, \Pi^{\sigma} \times \Sigma^{\sigma})}
{p^{\epsilon}(\Pi^{\sigma})\, p^{\eta}(\Sigma^{\sigma})\, \mathcal{G}(\omega_{\Sigma_f^{\sigma}})\,
p_{\infty}(\mu,\lambda)},
$$
where $\mathcal{G}(\omega_{\Sigma_f})$ is the Gauss sum attached to the central character of $\Sigma$. 
In particular,  
$$
L_f(1/2,\Pi \times \Sigma)\  
\sim_{{\mathbb Q}(\Pi,\Sigma)} \ 
p^{\epsilon}(\Pi)\, p^{\eta}(\Sigma)\, \mathcal{G}(\omega_{\Sigma_f})\, p_{\infty}(\mu,\lambda),
$$
where, by $\sim_{{\mathbb Q}(\Pi,\Sigma)}$, we mean up to an element of the number field which is the
compositum of the rationality fields ${\mathbb Q}(\Pi)$ and ${\mathbb Q}(\Sigma)$ of $\Pi$ and $\Sigma$ respectively. 
\end{thm}

The proof of the above theorem is based on a cohomological interpretation of the Rankin--Selberg zeta integral.  
That the Rankin--Selberg integral for 
${\rm GL}_n \times {\rm GL}_{n-1}$ admits a cohomological interpretation has been observed by several people. See 
especially, Schmidt \cite{schmidt}, 
Kazhdan, Mazur and Schmidt \cite{kazhdan-mazur-schmidt}, Mahnkopf \cite{mahnkopf-crelle}, \cite{mahnkopf-jussieu}, and 
Kasten and Schmidt \cite{kasten-schmidt}. However, for the application we have in mind, which is Deligne's conjecture for symmetric power $L$-functions, the above works are not suitable because of various assumptions made in those papers. We prove the above theorem while 
refining their techniques, especially those of Mahnkopf \cite{mahnkopf-jussieu}. The refinements are of two kinds:
\begin{enumerate}

\item We do not twist by a highly ramified character at places where $\Pi$ or $\Sigma$ is ramified as is done in 
\cite{mahnkopf-jussieu}. Instead, we 
use the observation that local special values are suitably rational (Proposition~\ref{prop:galois-localvalues}), and the possibly transcendental part of a global $L$-function is already captured by partial $L$-functions.  

\item The above papers are tailored toward constructing $p$-adic $L$-functions, in view of which there is a certain unipotent averaging that they consider at a prime where everything else is unramified. We consider the usual 
Rankin--Selberg integrals
without any such unipotent averaging. It is quite likely that our theorem above, plus a refinement of the 
period relations proved in our paper with Shahidi \cite{raghuram-shahidi-imrn}, can also be used to construct $p$-adic $L$-functions. 
\end{enumerate}

We briefly sketch the proof of Theorem~\ref{thm:rankin-selberg}. 
We make a very specific choice of Whittaker vectors for the two representations, and show that the Rankin--Selberg 
zeta integral of the cusp forms corresponding to these vectors, at $s =1/2$, can be interpreted as a pairing between certain cohomology classes. 
We choose a Whittaker vector $w_{\Pi_f}$ for the finite part $\Pi_f$, and  
let $\phi_{\Pi}$ be the cusp form corresponding
to $w_{\Pi_f}\otimes w_{\Pi_{\infty}}$, where $w_{\Pi_{\infty}}$ is a Whittaker vector at infinity.
Similarly, for a specific vector $w_{\Sigma_f}$, consider a cusp form $\phi_{\Sigma}$. 
The Rankin--Selberg integral at $1/2$ of these cusp forms, denoted $I(1/2,\phi_{\Pi},\phi_{\Sigma})$, is, up to 
controllable quantities, the $L$-value we are interested in (Proposition~\ref{prop:rankin-selberg}).
On the other hand, it may be interpreted as follows. To $w_{\Pi_f}$ is attached a cuspidal cohomology 
class $\vartheta_{\Pi}$ in $H^{b_n}_{\rm cusp}(F_n, \mathcal{M}_{\mu}^{\vee})$, where $b_n$ is the 
bottom degree of the cuspidal range for ${\rm GL}_n$, $F_n$ is a tentative notation for a locally symmetric space associated to ${\rm GL}_n$, and $\mathcal{M}_{\mu}^{\vee}$ is the sheaf on $F_n$ corresponding to the dual of the finite dimensional representation $M_{\mu}$ with highest weight $\mu$. Working with the dual of $M_{\mu}$ is only for convenience. Similarly, we have $\vartheta_{\Sigma} \in H^{b_{n-1}}_{\rm cusp}(F_{n-1}, \mathcal{M}_{\lambda}^{\vee})$. 
The hypothesis $\mu^{\vee} \succ \lambda$ implies that there is a canonical $G_{n-1}$ pairing 
$M_{\mu}^{\vee} \times M_{\lambda}^{\vee} \to {\mathbb Q}$.
The natural embedding ${\rm GL}_{n-1} \to {\rm GL}_n$ induces a proper map $\iota: F_{n-1} \to F_n$. We consider 
the wedge product $\vartheta_{\Sigma} \wedge \iota^*\vartheta_{\Pi}$, and observe that this happens to be a top-degree form on $F_{n-1}$ because $b_{n-1}+ b_n = {\rm dim}(F_{n-1})$; this numerical coincidence is at the heart of 
the proof. Integrating the top degree form over all of $F_{n-1}$ gives, after unravelling the definitions and using 
the calculation of the Rankin--Selberg integrals mentioned above, nothing but 
$L_f(1/2, \Pi \times \Sigma)\langle [\Sigma_{\infty}], [\Pi_{\infty}]\rangle$. This is the content of the 
{\it main identity} proved in Theorem~\ref{thm:main-identity}. The quantity 
$\langle [\Sigma_{\infty}], [\Pi_{\infty}]\rangle$, which depends only on the representations at infinity, is 
a linear combination of Rankin--Selberg integrals for `cohomological vectors'. One expects that it is nonzero.
We have not attempted a proof of this nonvanishing hypothesis, and so we need to assume its validity.  
The proof of Theorem~\ref{thm:rankin-selberg} follows since we can control algebraicity properties of the pairing of the classes $\vartheta_{\Pi}$ and $\vartheta_{\Sigma}$.

We now come to the second main theorem of this paper, which is to understand the behaviour of $L$-values under 
twisting by characters. We refer the reader to our papers with Shahidi \cite{raghuram-shahidi-aims} and 
\cite{raghuram-shahidi-imrn} for motivational background for such results. We note that results of this kind are
predicted by the results and conjectures of Blasius \cite{blasius2} and Panchishkin \cite{panchishkin}, both of whom independently calculated the behaviour of Deligne's periods attached to a motive upon twisting by Artin motives. 
Our second theorem is:

\begin{thm}
\label{thm:twisted}
Let $\Pi$ and $\Sigma$ be as in Theorem~\ref{thm:rankin-selberg}. Let $\xi$ be an even Dirichlet character which
we identify with the corresponding Hecke character of ${\mathbb Q}$. We have
$$
L_f(1/2, (\Pi\otimes\xi) \times \Sigma) \sim_{{\mathbb Q}(\Pi,\Sigma,\xi)} 
\mathcal{G}(\xi_f)^{n(n-1)/2}L_f(1/2,\Pi \times \Sigma),
$$
where, by $\sim_{{\mathbb Q}(\Pi,\Sigma,\xi)}$, we mean up to an element of the number field 
${\mathbb Q}(\Pi,\Sigma,\xi)$ which is the compositum
of the rationality fields ${\mathbb Q}(\Pi)$, ${\mathbb Q}(\Sigma)$ and ${\mathbb Q}(\xi)$ of $\Pi$, $\Sigma$, and 
$\xi$ respectively. Moreover, if $L_f(1/2, \Pi \times \Sigma) \neq 0$, then for any $\sigma \in {\rm Aut}({\mathbb C})$
we have 
$$
\sigma \left( \frac{L_f(1/2, (\Pi\otimes\xi) \times \Sigma)}
{\mathcal{G}(\xi_f)^{n(n-1)/2}\, L_f(1/2,\Pi \times \Sigma)}\right) = 
\frac{L_f(1/2, (\Pi^{\sigma}\otimes\xi^{\sigma}) \times \Sigma^{\sigma})}
{\mathcal{G}(\xi_f^{\sigma})^{n(n-1)/2}\, L_f(1/2,\Pi^{\sigma} \times \Sigma^{\sigma})}.
$$
\end{thm}

We remark that our proof of Theorem~\ref{thm:twisted} uses Theorem~\ref{thm:rankin-selberg}, and so  
is subject to the assumption made in Hypothesis~\ref{hypo:nonvanishing}.

We now describe an application of Theorem~\ref{thm:rankin-selberg} to the special values of symmetric power 
$L$-functions. Let $\varphi$ be a primitive holomorphic cusp form on the upper half plane of weight $k$, for 
$\Gamma_0(N)$, with nebentypus character $\omega$. We denote this as $\varphi \in S_k(N,\omega)_{\rm prim}$. 
For any integer $r \geq 1$, consider the  
$r$-th symmetric power $L$-function $L(s, {\rm Sym}^r \varphi, \xi)$ attached to $\varphi$, twisted by a 
Dirichlet character $\xi$. The sign of $\xi$ is defined as $\epsilon_{\xi} = \xi(-1)$.
(We will think of $\xi$ as a Hecke character of ${\mathbb Q}$.)
Our final theorem in this
paper gives an algebraicity theorem for certain critical values of such $L$-functions when $r$ is an odd integer
$\leq 7$.

\begin{thm}
\label{thm:sym-357}
Let $\varphi \in S_k(N,\omega)_{\rm prim}$, $n$ a positive integer $\leq 4$, and $\xi$ a Dirichlet character.
Let $m$ be the critical integer for $L_f(s,{\rm Sym}^{2n-1}(\varphi),\xi)$ given by: 
\begin{enumerate}
\item If $k$ is even, then we assume $k \geq 4$ and let $m = ((2n-1)(k-1)+3)/2$.
\item If $k$ is odd, then we assume $k \geq 3$ and let $m = ((2n-1)(k-1)+2)/2$. 
\end{enumerate} 
There exists nonzero complex numbers $p^{\epsilon}(\varphi, 2n-1)$ depending on the form $\varphi$, the 
integer $n$, and a sign $\epsilon \in \{\pm \}$, and there exists a nonzero complex number $p(m,k)$ depending on
the critical point $m$ and the weight $k$, such that for any $\sigma \in {\rm Aut}({\mathbb C})$ we have 
$$
\sigma\left(\frac{L_f(m, {\rm Sym}^{2n-1}(\varphi), \xi)}
{p^{\epsilon_{\xi}}(\varphi, 2n-1)\, p(m,k)\, \mathcal{G}(\xi_f)^{n}}\right) = 
\frac{L_f(m, {\rm Sym}^{2n-1}(\varphi^{\sigma}), \xi^{\sigma})}
{p^{\epsilon_{\xi}}(\varphi^{\sigma}, 2n-1) \, p(m,k)\, \mathcal{G}(\xi_f^{\sigma})^{n}}.
$$
In particular, 
$$
L_f(m, {\rm Sym}^{2n-1}(\varphi), \xi)\  
\sim_{{\mathbb Q}(\varphi,\xi)} \ 
p^{\epsilon_{\xi}}(\varphi, 2n-1)\, p(m,k)\, \mathcal{G}(\xi_f)^{n},
$$
where, by $\sim_{{\mathbb Q}(\varphi, \xi)}$, we mean up to an element of the number field generated 
by the Fourier coefficients of $\varphi$ and the values of $\xi$. 

Further, if we assume Langlands' functoriality, in as much as assuming that the transfer of automorphic representations 
holds for the $L$-homomorphism ${\rm Sym}^l : {\rm GL}_2({\mathbb C}) \to {\rm GL}_{l+1}({\mathbb C})$ for all 
integers $l \geq 1$, then the above statements about critical values holds for all odd positive integers $2n-1$.
\end{thm}

We remark that our proof of Theorem~\ref{thm:sym-357} uses Theorem~\ref{thm:rankin-selberg}, and so  
is subject to the assumption made in Hypothesis~\ref{hypo:nonvanishing}. Let $\pi(\varphi)$ be 
the cuspidal automorphic representation attached to $\varphi$, and let ${\rm Sym}^r(\pi(\varphi))$ denote 
the $r$-th symmetric power transfer; it is known to exist for $r \leq 4$ by the work of 
Gelbart and Jacquet \cite{gelbart-jacquet},  
Kim and Shahidi \cite{kim-shahidi-annals}, and Kim \cite{kim}.
The proof of Theorem~\ref{thm:sym-357} is obtained 
by recursively applying Theorem~\ref{thm:rankin-selberg} to the pair $({\rm Sym}^n(\pi(\varphi)), 
{\rm Sym}^{n-1}(\pi(\varphi)))$, up to appropriate twisting (Proposition~\ref{prop:twisting}).
The critical point $m$ that we consider is on the right edge of symmetry when $k$ is odd, and is one unit to the right of the center of symmetry when $k$ is even. The quantity $p^{\epsilon}(\varphi,2n-1)$ is a combination of periods attached to ${\rm Sym}^r(\pi(\varphi))$ for $r \leq n$, and  
the quantity $p(m,k)$ is a combination of some of the $p_{\infty}(\mu,\lambda)$ that show up in 
Theorem~\ref{thm:rankin-selberg}. We expect that our results are compatible with Deligne's conjecture 
\cite[\S7.8]{deligne}, in view of which, 
we formulate Conjecture~\ref{con:period-relations} relating the periods attached to the representations $\Pi$ and $\Sigma$ as above, and Deligne's periods $c^{\pm}(M)$, where $M$ is the tensor product of the conjectural motives 
$M(\Pi)$ and $M(\Sigma)$. 

Finally, we note that in this paper we have considered only one critical point for any given $L$-function. In joint work with G\"unter Harder, we are investigating the algebraicity properties of ratios of successive critical values for 
the Rankin--Selberg $L$-functions considered above. This will then give us algebraicity results for ratios of 
successive critical values for the odd symmetric power $L$-functions considered above. 
The results of this investigation will appear elsewhere. On an entirely different note, we mention the recent work 
of Gan, Gross and Prasad \cite{gan-gross-prasad} on generalizations of the Gross-Prasad conjectures; they too are interested in the central critical value $L(1/2,\Pi\times\Sigma)$, albeit, from a different perspective.

\bigskip

{\small
{\it Acknowledgements:} It is a pleasure to thank Don~Blasius, G\"unter~Harder, Michael~Harris, Paul~Garrett, Ameya~Pitale and Freydoon~Shahidi for their interest and helpful discussions. I thank Jishnu~Biswas and 
Vishwambar~Pati for clarifying some topological details. Much of this work was carried out during a visit to the Max Planck Institute in 2008; I gratefully acknowledge their invitation and thank MPI for providing an excellent atmosphere.}

\section{Notations, conventions, and preliminaries}
\label{sec:prelims}

The algebraic group ${\rm GL}_n$ over ${\mathbb Q}$ will be denoted as $G_n$. Let $B_n = T_nN_n$ stand for the standard Borel subgroup of $G_n$ of all upper triangular matrices, $N_n$ the unipotent radical of $B_n$, and $T_n$ the diagonal 
torus. The center of $G_n$ will be denoted by $Z_n$. The identity element of $G_n$ will be denoted $1_n$.

We let $X^+(T_n)$ stand for the set of dominant (with respect to $B_n$) integral weights of $T_n$, and for 
$\mu \in X^+(T_n)$ we denote by $M_{\mu}$ the irreducible representation of 
$G_n({\mathbb C})$ with highest weight $\mu$. Note that $M_{\mu}$ is defined over ${\mathbb Q}$. Let $M_{\mu}^{\vee}$
denote the contragredient of $M_{\mu}$ and define the dual weight $\mu^{\vee}$ by $M_{\mu}^{\vee} = M_{\mu^{\vee}}$.
We let $X_0^+(T_n)$ stand for the subset of $X^+(T_n)$ consisting of 
pure weights \cite[(3.1)]{mahnkopf-jussieu}. If $\mu = (\mu_1,\dots,\mu_n) \in X^+(T_n)$ and 
$\lambda = (\lambda_1,\dots,\lambda_{n-1}) \in X^+(T_{n-1})$ then by $\mu \succ \lambda$ we mean the condition 
$\mu_1 \geq \lambda_1 \geq \mu_2 \geq \lambda_2 \geq \cdots \geq \lambda_{n-1} \geq \mu_n$, which ensures that 
$M_{\lambda}$ appears in the restriction to $G_{n-1}$ of $M_{\mu}$; in fact it appears with multiplicity one. 

We let ${\mathbb A}$ stand for the 
ad\`ele ring of ${\mathbb Q}$, and ${\mathbb A}_f$ the ring of finite ad\`eles. 
Following Borel--Jacquet \cite[\S4.6]{borel-jacquet}, we say an irreducible 
representation of $G_n({\mathbb A})$ is automorphic if it is isomorphic to an 
irreducible subquotient of the representation of $G_n({\mathbb A})$ on its
space of automorphic forms. We say an automorphic representation is cuspidal 
if it is a subrepresentation of the representation of $G_n({\mathbb A})$ on 
the space of cusp forms $\mathcal{A}_{\rm cusp}(G_n({\mathbb Q})\backslash G_n({\mathbb A}))$. 
The subspace of cusp forms realizing $\Pi$ will be denoted $V_{\Pi}$. 
For an automorphic representation $\Pi$ of $G_n({\mathbb A})$, we have
$\Pi = \Pi_{\infty} \otimes \Pi_f$, where $\Pi_{\infty}$ is a representation of $G_{n,\infty} = G_n({\mathbb R})$,
and $\Pi_f = \otimes_{v \neq \infty} \Pi_v$
is a representation of $G_n({\mathbb A}_f)$. The central character of any irreducible representation $\Theta$ will be denoted $\omega_{\Theta}$. The finite part of a global $L$-function is denoted $L_f(s,\Pi)$, and for any place $v$ the local $L$-factor at $v$ is denoted $L(s, \Pi_v)$.

We will let $K_{n,\infty}$ stand for ${\rm O}(n)Z_n({\mathbb R})$; it is the thickening of the 
maximal compact subgroup of $G_{n,\infty}$ by the center $Z_{n,\infty}$. Let $K_{n,\infty}^0$ be 
the topological connected component of $K_{n,\infty}$. For any group $\mathfrak{G}$ we will let 
$\pi_0(\mathfrak{G})$ stand for the group of connected components. 
We will identify $\pi_0(G_n) = \pi_0(K_{n,\infty}) \simeq  \{\pm 1 \} = \{\pm \}$. Note that 
$\delta_n = {\rm diag}(-1,1,\dots,1)$ represents the nontrivial element in $\pi_0(K_{n,\infty})$, and if $n$ is odd, the element $-1_n$ also represents this nontrivial element. We will further identify $\pi_0(K_{n,\infty})$ with its
character group $\pi_0(K_{n,\infty}){}^{\widehat{}}$. Let $K_{n,\infty}^1 = {\rm SO}(n)$. 

Let $\iota : G_{n-1} \to G_n$ be the map $g \mapsto \left(\begin{smallmatrix}g & \\ & 1\end{smallmatrix}\right)$. 
Then $\iota$ induces a map at the level of local and global groups, and between appropriate symmetric spaces 
of $G_{n-1}$ and $G_n$, all of which will also be denoted by $\iota$ again; we hope that this will cause no confusion. 
The pullback (of a subset, a function, a differential form, or a cohomology class) via $\iota$ will be 
denoted $\iota^*$.

Fix a global measure $dg$ on $G_n({\mathbb A})$ which is a product of local measures $dg_v$. The local measures
are normalized as follows: for a finite place $v$, if $\mathcal{O}_v$ is the ring of integers of ${\mathbb Q}_v$, then 
we assume that ${\rm vol}(G_n(\mathcal{O}_v)) = 1$; and at infinity assume that ${\rm vol}(K_{n,\infty}^1) = 1$. 

For a Dirichlet character $\chi$ modulo an integer $N$, following Shimura \cite{shimura1}, we define its Gauss sum 
$\mathfrak{g}(\chi)$ as the Gauss sum of its associated primitive character, say $\chi_0$ of conductor $c$, where
$\mathfrak{g}(\chi_0) = \sum_{a=0}^{c-1}\chi_0(a)e^{2\pi i a/c}$. For a Hecke character $\xi$ of ${\mathbb Q}$, by which we mean a continuous homomorphism $\xi: {\mathbb Q}^*\backslash {\mathbb A}^{\times} \to {\mathbb C}^*$, 
following Weil \cite[Chapter VII, \S7]{weil}, we define the Gauss sum of $\xi$ as follows: We let $\mathfrak{c}$ stand for the conductor ideal of $\xi_f$. 
We fix, once and for all, an additive character $\psi$ of ${\mathbb Q} \backslash {\mathbb A}$, as in Tate's thesis, namely, $\psi(x) = e^{2\pi i \Lambda(x)}$ with the $\Lambda$ as defined in 
\cite[\S 4.1]{tate}. Let $y = (y_v)_{v \neq \infty} 
\in {\mathbb A}_f^{\times}$ be such that 
${\rm ord}_v(y_v) = -{\rm ord}_v(\mathfrak{c})$.  The Gauss sum of $\xi$ is 
defined as  $\mathcal{G}(\xi_f,\psi_f,y) = \prod_{v \neq \infty} \mathcal{G}(\xi_v,\psi_v,y_v)$
where the local Gauss sum $\mathcal{G}(\xi_v,\psi_v,y_v)$ is defined as
$$
\mathcal{G}(\xi_v,\psi_v,y_v) = \int_{\mathcal{O}_v^{\times}} \xi_v(u_v)^{-1}\psi_v(y_vu_v)\, du_v.
$$
For almost all $v$, where everything in sight is unramified, we have $\mathcal{G}(\xi_v,\psi_v,y_v)=1$, and 
for all $v$ we have $\mathcal{G}(\xi_v,\psi_v,y_v) \neq 0$. (See, for example, Godement \cite[Eqn. 1.22]{godement}.)
Note that, unlike Weil, we do not normalize the Gauss sum to make it have absolute value one and we do not have any factor at infinity. Suppressing the dependence on $\psi$ and $y$, we denote $\mathcal{G}(\xi_f,\psi_f,y)$ simply 
by $\mathcal{G}(\xi_f)$.
To have the functional equations of the $L$-functions of a Dirichlet character $\chi$ and the corresponding
Hecke character $\xi$ to {\it look the same} we need the Gauss sums to be defined as above; compare Neukirch \cite[Chapter VII, Theorem 2.8]{neukirch} with Weil \cite[Chapter VII, Theorem 5]{weil}.

In our paper with Shahidi \cite{raghuram-shahidi-imrn} we defined the Gauss sum $\gamma(\xi_f)$ of a 
Hecke character $\xi$ as $\mathcal{G}(\xi_f^{-1})$. 
Since this article crucially uses the results of \cite{raghuram-shahidi-imrn} it is helpful to record 
the following details that we will repeatedly use: 
Lemma 4.3 of \cite{raghuram-shahidi-imrn} now reads as 
\begin{equation}
\label{eqn:variation1}
\sigma(\xi_f(t_{\sigma})) = \sigma(\mathcal{G}(\xi_f))/\mathcal{G}(\xi_f^{\sigma}), 
\end{equation}
and Theorem 4.1(1) of \cite{raghuram-shahidi-imrn} now reads as
\begin{equation}
\label{eqn:variation2}
\sigma\left(
\frac{p^{\epsilon \cdot \epsilon_{\xi}}(\Pi_f\otimes\xi_f)}
{\mathcal{G}(\xi_f)^{n(n-1)/2}\,p^{\epsilon}(\Pi_f)} \right) 
=
\left(\frac{p^{\epsilon\cdot\epsilon_{\xi^{\sigma}}}(\Pi_f^{\sigma}\otimes\xi_f^{\sigma})}
{\mathcal{G}(\xi_f^{\sigma})^{n(n-1)/2}\,p^{\epsilon}(\Pi_f^{\sigma})} \right),
\end{equation}
where $\Pi$ is a regular algebraic cuspidal automorphic representation of ${\rm GL}_n({\mathbb A})$ and 
$\xi$ is an algebraic Hecke character of ${\mathbb Q}$.

\section{Rankin-Selberg $L$-functions for ${\rm GL}_n \times {\rm GL}_{n-1}$}
\label{sec:rankin-selberg}

\subsection{The global integral}
\label{sec:global-integral}

\subsubsection{}
We consider the Rankin--Selberg zeta integrals for ${\rm GL}_n \times {\rm GL}_{n-1}$. 
(See the works of Jacquet, Piatetski-Shapiro and Shalika \cite{jacquet-shalika-ajm}, 
\cite{jacquet-ps-shalika-ajm83}. We roughly follow the notation in Cogdell's expository article 
\cite{cogdell-notes}.)
Let $\Pi$ (resp., $\Sigma$) be a cuspidal automorphic representation of $G_n({\mathbb A})$
(resp., $G_{n-1}({\mathbb A})$). Let $\phi \in V_{\Pi}$ and $\phi' \in V_{\Sigma}$ be cusp forms. 
The zeta integral we are interested in is given by
$$
I(s, \phi, \phi') = \int_{G_{n-1}({\mathbb Q})\backslash G_{n-1}({\mathbb A})}
\phi(\iota(g))\phi'(g)|{\rm det}(g)|^{s - 1/2}\, dg.
$$
Since the cusp forms $\phi$ and $\phi'$ are rapidly decreasing, the above integral converges for all 
$s \in {\mathbb C}$. Suppose that $w \in W(\Pi,\psi)$ and $w' \in W(\Sigma, \psi^{-1})$ are global Whittaker functions corresponding to $\phi$ and $\phi'$, respectively; recall that $\psi$ is a 
nontrivial additive character ${\mathbb Q}\backslash {\mathbb A}$. After the usual unfolding, one has 
$$
I(s,\phi,\phi') = \Psi(s,w,w') := \int_{N_{n-1}({\mathbb A})\backslash G_{n-1}({\mathbb A})}
w(\iota(g))w'(g)|{\rm det}(g)|^{s - 1/2}\, dg.
$$
The integral $\Psi(s,w,w')$ converges for ${\rm Re}(s) \gg 0$. 
Let $w = \otimes w_v$ and $w' = \otimes w'_v$, then $\Psi(s,w,w') := \otimes \Psi_v(s,w_v,w'_v)$ for 
${\rm Re}(s) \gg 0$, where the local 
integral $\Psi_v$ is given by a similar formula. Recall that the local integral $\Psi_v(s,w_v,w'_v)$ converges 
for ${\rm Re}(s) \gg 0$ and has a meromorphic continuation to all of ${\mathbb C}$; see 
\cite[Proposition 6.2]{cogdell-notes} for $v < \infty$, and for $v = \infty$ see 
\cite[Theorem 1.2(i)]{cogdell-ps}.
We will choose the local Whittaker functions carefully so that
the integral $I(1/2,\phi,\phi')$ computes the special value $L_f(1/2, \Pi\times\Sigma)$ up to quantities which are under control, in the sense that they will be ${\rm Aut}({\mathbb C})$-equivariant. Before making this choice of vectors, we review some ingredients.

\subsubsection{Action of ${\rm Aut}({\mathbb C})$ on Whittaker models}
Consider the cyclotomic character
$$
\begin{array}{llllllc}
{\rm Aut}({\mathbb C}/{\mathbb Q}) & \to & 
{\rm Gal}(\overline{\mathbb Q}/{\mathbb Q}) & \to&
{\rm Gal}({\mathbb Q}(\mu_{\infty})/{\mathbb Q}) & \to &
\widehat{{\mathbb Z}}^{\times} \simeq \prod_p {\mathbb Z}_p^{\times} \\
\sigma & \mapsto & \sigma |_{\overline{\mathbb Q}} & \mapsto & 
\sigma |_{{\mathbb Q}(\mu_{\infty})} & \mapsto & t_{\sigma} 
\end{array}
$$
The element $t_{\sigma}$ at the end can
be thought of as an element of ${\mathbb A}_f^{\times} = {\mathbb I}_f$. 
Let $t_{\sigma,n}$ denote the diagonal matrix 
${\rm diag}(t_{\sigma}^{-(n-1)}, t_{\sigma}^{-(n-2)},\dots,1)$ regarded as an 
element of  ${\rm GL}_n({\mathbb A}_f)$. 
For $\sigma \in {\rm Aut}({\mathbb C})$ and 
$w \in W(\Pi_f,\psi_f)$, define the function ${}^{\sigma}\!w$ by
$$
{}^{\sigma}\!w(g_f) = \sigma(w(t_{\sigma,n}g_f))
$$
for all $g_f \in {\rm GL}_n({\mathbb A}_f)$. Note that this action makes sense locally, by replacing
$t_{\sigma}$ by $t_{\sigma, v}$. Further, if $\Pi_v$ is unramified, then the spherical vector is 
mapped to the spherical vector under $\sigma$. This makes the local and global actions compatible. 
For more details, see \cite[\S 3.2]{raghuram-shahidi-imrn}. (See also \S\ref{sec:different-rational-structures} where we discuss other possible actions of ${\rm Aut}({\mathbb C})$ on Whittaker models.)

\subsubsection{Normalized new vectors}
\label{sec:newvectors}
We review some details about local new (or essential) vectors \cite{jacquet-ps-shalika}.
Just for this paragraph, let $F$ be a non-archimedean local field, $\mathcal{O}_F$ the 
ring of integers of $F$, and $\mathcal{P}_F$ the maximal ideal of $\mathcal{O}_F$. 
Let $(\pi,V)$ be an irreducible admissible generic representation of ${\rm GL}_n(F)$. Let $K_n(m)$ be the 
`mirahoric subgroup' of ${\rm GL}_n(\mathcal{O}_F)$ consisting of all matrices whose last row 
is congruent to $(0,\dots,0,*)$ modulo $\mathcal{P}_F^m$.  
Let $V_m := \{v \in V \ | \ \pi(k)v = 
\omega_{\pi}(k_{n,n})v, \forall k \in K_n(m) \}$. Let $\mathfrak{f}(\pi)$ be the least non-negative integer $m$ for which 
$V_m \neq (0)$. One knows that $\mathfrak{f}(\pi)$ is the conductor of $\pi$ (in the sense of epsilon factors), and that 
$V_{\mathfrak{f}(\pi)}$ is one-dimensional. Any vector in $V_{\mathfrak{f}(\pi)}$ is called a {\it new vector} of 
$\pi$. 
Fix a nontrivial additive character $\psi$ of $F$, and assume that $V = W(\pi,\psi)$ is the Whittaker model for $\pi$.
If $\pi$ is unramified, i.e., $\mathfrak{f}(\pi) = 0$, 
then we fix a specific new vector called the {\it spherical vector}, which we denote $w_{\pi}^{\rm sp}$, 
normalized such that $w_{\pi}^{\rm sp}(1_n) = 1.$ More generally, for 
any $\pi$, amongst all new vectors, there is a distinguished vector, called the {\it essential vector}, which we denote as $w_{\pi}^{\rm ess}$, characterized by the property that for any irreducible unramified generic representation $\rho$ of ${\rm GL}_{n-1}(F)$ one has 
$$
\Psi(s, w_{\pi}^{\rm ess}, w_{\rho}^{\rm sp}) = 
\int_{N_{n-1}(F)\backslash G_{n-1}(F)}
w_{\pi}^{\rm ess}(\iota(g))w_{\rho}^{\rm sp}(g)|{\rm det}(g)|^{s - 1/2}\, dg = L(s, \pi \times \rho).
$$
We note that if $\pi$ is unramified then $w_{\pi}^{\rm ess} = w_{\pi}^{\rm sp}$. 
Although the essential vector has the above nice analytic property, it does not, in general, have good arithmetic properties in the sense that essential vectors are not ${\rm Aut}({\mathbb C})$-equivariant. For this equivariance, following Mahnkopf, using 
\cite[Lemma 1.3.2]{mahnkopf-jussieu}, we fix the following normalization. This lemma says that given $\pi$ 
there exists $t_{\pi} \in T_n(F)$ such that a new vector for $\pi$ is nonvanishing on $t_{\pi}$. Note that 
necessarily $t_{\pi} \in T_n^+(F)$, i.e., if $t_{\pi} = {\rm diag}(t_1,t_2,\dots,t_n)$ then $t_it_{i+1}^{-1} \in \mathcal{O}_F$ for all $1 \leq i \leq n-1$. 
We let $w_{\pi}^0$ be the new vector normalized such that 
$w_{\pi}^0(t_{\pi}) = 1$. If $\pi$ is unramified then we may and will take $t_{\pi} = 1_n$, and so
$w_{\pi}^0 = w_{\pi}^{\rm ess} = w_{\pi}^{\rm sp}$. For any $\sigma \in {\rm Aut}({\mathbb C})$ we may and will take 
$t_{\pi^{\sigma}} = t_{\pi}$. Then it is easy to see that ${}^{\sigma}\!w_{\pi}^0 = w_{\pi^{\sigma}}^0$. We define the
scalar $c_{\pi} \in {\mathbb C}^*$ by $w_{\pi}^0 = c_{\pi}w_{\pi}^{\rm ess}$, i.e., 
$c_{\pi} = w_{\pi}^{\rm ess}(t_{\pi})^{-1}.$

\subsubsection{Choice of Whittaker vectors and cusp forms}
\label{subsec:whittaker}
We now go back to global notation and choose global Whittaker vectors $w_{\Pi} = \otimes_v w_{\Pi,v} \in W(\Pi,\psi)$ and $w_{\Sigma} = \otimes_v w_{\Sigma,v} \in W(\Sigma, \psi^{-1})$ as follows. 
Let $S_{\Sigma}$ be the set of finite places $v$ where $\Sigma_v$ is unramified.  
\begin{enumerate}
\item If $v \notin S_{\Sigma}\cup \{\infty \}$,  we let $w_{\Pi, v} = w_{\Pi_v}^0$, and 
$w_{\Sigma, v} = w_{\Sigma_v}^{\rm sp}$. 
\item If $v \in S_{\Sigma}$, we let $w_{\Sigma, v} = w_{\Sigma_v}^0$, and let $w_{\Pi, v}$ be the unique Whittaker 
function whose restriction to $G_{n-1}({\mathbb Q}_v)$ is supported on 
$N_{n-1}({\mathbb Q}_v)t_{\Sigma_v}K_{n-1}(\mathfrak{f}(\Sigma_v))$, and on this double coset it is given by 
$w_{\Pi, v}(ut_{\Sigma_v}k) = \psi(u)\omega_{\Sigma_v}^{-1}(k_{n-1,n-1})$, 
for all $u \in N_{n-1}({\mathbb Q}_v)$ and for all $k \in K_{n-1}(\mathfrak{f}(\Sigma_v))$. 
\item If $v = \infty$, we let $w_{\Pi,\infty}$ and $w_{\Sigma, \infty}$ be arbitrary nonzero vectors. 
(Later, these will be cohomological vectors.) 
\end{enumerate}
Let $w_{\Pi_f} = \otimes_{v \neq \infty} w_{\Pi,v}$ and $w_{\Pi} = w_{\Pi_{\infty}} \otimes w_{\Pi_f}$. 
Similarly, let $w_{\Sigma_f} = \otimes_{v \neq \infty} w_{\Sigma,v}$ and 
$w_{\Sigma} = w_{\Sigma_{\infty}} \otimes w_{\Sigma_f}$. Let $\phi_{\Pi}$ 
(resp., $\phi_{\Sigma}$) be the cusp form corresponding to $w_{\Pi}$ (resp., $w_{\Sigma}$).

\subsubsection{Rankin--Selberg $L$-functions}

\begin{prop}
\label{prop:rankin-selberg}
We have 
{\small
$$
I(1/2, \phi_{\Pi}, \phi_{\Sigma}) \ = \  
\frac{\Psi_{\infty}(1/2, w_{\Pi_{\infty}}, w_{\Sigma_{\infty}})\, {\rm vol}(\Sigma)\, 
\prod_{v \notin S_{\Sigma} \cup \{\infty\}} c_{\Pi_v}}
{\prod_{v \in S_{\Sigma}} L(1/2, \Pi_v \times \Sigma_v)} 
L_f(1/2, \Pi \times \Sigma),
$$}
where ${\rm vol}(\Sigma) = \prod_{v \in S_{\Sigma}} {\rm vol}(K_{n-1}(\mathfrak{f}(\Sigma_v)) \in {\mathbb Q}^*$.
\end{prop}

\begin{proof}
\begin{eqnarray*}
I(s, \phi_{\Pi}, \phi_{\Sigma}) & = & 
\int_{G_{n-1}({\mathbb Q})\backslash G_{n-1}({\mathbb A})}
\phi_{\Pi}(\iota(g))\phi_{\Sigma}(g)|{\rm det}(g)|^{s - 1/2}\, dg, \ \ (\forall s \in {\mathbb C})\\
& = & 
\int_{N_{n-1}({\mathbb A})\backslash G_{n-1}({\mathbb A})}
w_{\Pi}(\iota(g))w_{\Sigma}(g)|{\rm det}(g)|^{s - 1/2}\, dg, \ \ ({\rm Re}(s) \gg 0) \\
& = & 
\prod_v \int_{N_{n-1}({\mathbb Q}_v)\backslash G_{n-1}({\mathbb Q}_v)}
w_{\Pi_v}(\iota(g_v))w_{\Sigma_v}(g_v)|{\rm det}(g_v)|^{s - 1/2}\, dg_v \\
& = & 
\Psi_{\infty}(s, w_{\Pi_{\infty}}, w_{\Sigma_{\infty}})
\prod_{v \notin S_{\Sigma} \cup \{\infty\}} c_{\Pi_v}L(s, \Pi_v \times \Sigma_v)
\prod_{v \in S_{\Sigma}} {\rm vol}(K_{n-1}(\mathfrak{f}(\Sigma_v)). 
\end{eqnarray*}
The last equality is because of our specific choice of Whittaker vectors. 
Multiplying and dividing by the local factors for $v \in S_{\Sigma}$ we get
$$ 
I(s, \phi_{\Pi}, \phi_{\Sigma}) = 
\frac{\Psi_{\infty}(s, w_{\Pi_{\infty}}, w_{\Sigma_{\infty}})\, {\rm vol}(\Sigma)\, 
\prod_{v \notin S_{\Sigma} \cup \{\infty\} } c_{\Pi_v}}
{\prod_{v \in S_{\Sigma}} L(s, \Pi_v \times \Sigma_v)} 
L_f(s, \Pi \times \Sigma), \ \ ({\rm Re}(s) \gg 0).
$$
The left hand side is defined for all $s$, and the right hand side has a meromorphic continuation to all of
${\mathbb C}$. Hence we get equality at $s = 1/2$. Since $c_{\Pi_v} = 1$ if $\Pi_v$ is unramified, the 
product $\prod_{v \notin S_{\Sigma} \cup \{\infty\} } c_{\Pi_v}$ is really a finite product. 
\end{proof}

\subsection{Cohomological interpretation of the integral}

We interpret the Rankin--Selberg integral $I(1/2, \phi_{\Pi}, \phi_{\Sigma})$ 
in terms of Poincar\'e duality. More precisely, the vector $w_{\Pi_f}$ will correspond to a cohomology class $\vartheta_{\Pi}$
in degree $b_n$ (the bottom degree of the cuspidal range for $G_n$) on a locally symmetric space tentatively denoted $F_n$ for ${\rm GL}_n$, and similarly $w_{\Sigma_f}$ will correspond to a class $\vartheta_{\Sigma}$ in degree 
$b_{n-1}$ on $F_{n-1}$. 
These classes, after dividing by certain periods, have good rationality properties. 
We pull back $\vartheta_{\Pi}$ along 
the proper map $\iota: F_{n-1} \to F_n$, and wedge (or cup) with $\vartheta_{\Sigma}$, to give a top degree class on 
$F_{n-1}$. It is of top degree 
because $b_n + b_{n-1} = {\rm dim}(F_{n-1})$; this numerical coincidence is at the heart of other works
on Rankin--Selberg $L$-functions. (See, for example, 
Kazhdan, Mazur and Schmidt \cite[\S 1]{kazhdan-mazur-schmidt} or Mahnkopf \cite[p. 616]{mahnkopf-jussieu}.)
Integrating this form on $F_{n-1}$, which indeed is the Rankin--Selberg integral of the previous section, is nothing but applying a linear functional to cohomology in top degree, and the point is that this functional is that obtained from pairing with a certain cycle (constructed as in Mahnkopf \cite[5.1.1]{mahnkopf-jussieu}, which in turn is a generalization of Harder's construction  \cite{harder} for ${\rm GL}_2$). 
Interpreting the integral, and hence the special values of $L$-functions, as a cohomological pairing permits us to study arithmetic properties of the special values, since this pairing is Galois equivariant. We now make all this precise.

\subsubsection{The periods}
\label{subsec:periods}
We assume that the reader is familiar with our paper with Shahidi \cite{raghuram-shahidi-imrn}, and especially the definition of the periods attached to regular algebraic cuspidal representations. We review the very 
basic ingredients here, and refer the reader to \cite{raghuram-shahidi-imrn} for all finer details. 
See especially \cite[Definition/Proposition 3.3]{raghuram-shahidi-imrn}. We also use the same notation as in 
that paper, with just one exception that we mention in the next paragraph.  

Assume now that the cuspidal representation $\Pi$ (resp., $\Sigma$) is regular and algebraic. A consequence is that 
there is a weight $\mu \in X^+(T_n)$ (resp., $\lambda \in X^+(T_{n-1})$) 
such that $\Pi \in {\rm Coh}(G_n,\mu^{\vee})$ (resp., $\Sigma \in {\rm Coh}(G_{n-1},\lambda^{\vee})$).
The weight $\mu$ is a dominant integral weight which is {\it pure} (by \cite[Lemme de puret\'e 4.9]{clozel}), 
i.e., if $\mu = (\mu_1,\dots,\mu_n)$, then there is an integer ${\rm wt}(\mu)$ such that 
$\mu_i + \mu_{n-i+1} = {\rm wt}(\mu)$. We will denote by $X^+_0(T_n)$ the set of dominant integral pure weights
for $T_n$. Similarly, $\lambda \in X_0^+(T_{n-1})$.
Let $\epsilon \in \{\pm \} \simeq (K_{n,\infty}/K_{n, \infty}^0){}^{\widehat{}}$ be a sign, which can be 
arbitrary if $n$ is even, and is uniquely determined by $\Pi$ if $n$ is odd. (If $n$ is odd then 
$\epsilon = \omega_{\Pi_{\infty}}(-1)(-1)^{{\rm wt}(\mu)/2}$, which is the central character of 
$\Pi_{\infty}\otimes M_{\mu}^{\vee}$ at $-1$.) Such an $\epsilon$ is called a 
permissible sign for $\Pi$. We define $b_n = n^2/4$ if $n$ is even, and $b_n = (n^2-1)/4$ if 
$n$ is odd. We have a map 
$$
\mathcal{F}_{\Pi_f,\epsilon,[\Pi_{\infty}]} : W(\Pi_f) \to H^{b_n}(\mathfrak{g}_{\infty}, K_{\infty}^0; 
V_{\Pi}\otimes M_{\mu}^{\vee})(\epsilon).
$$
We note that difference in notation mentioned above: a choice of generator for the one-dimensional ${\mathbb C}$-vector space $H^{b_n}(\mathfrak{g}_{\infty}, K_{\infty}^0; \Pi_{\infty}\otimes M_{\mu}^{\vee})(\epsilon)$ which was denoted 
${\bf w}_{\infty}$ in \cite{raghuram-shahidi-imrn}, will be denoted by $[\Pi_{\infty}]$ in this paper. 
The map $\mathcal{F}_{\Pi_f,\epsilon,[\Pi_{\infty}]}$ is a $G_n({\mathbb A}_f)$-equivariant map between 
irreducible modules, both of which have ${\mathbb Q}(\Pi)$-rational structures that are unique up to homotheties. 
For the action of 
${\rm Aut}({\mathbb C})$ and the rational structure on the Whittaker model $W(\Pi_f)$ 
see \cite[\S 3.2]{raghuram-shahidi-imrn}, and on $H^{b_n}(\mathfrak{g}_{\infty}, K_{\infty}^0; 
V_{\Pi}\otimes M_{\mu}^{\vee})(\epsilon)$ see \cite[\S 3.3]{raghuram-shahidi-imrn}.
The period $p^{\epsilon}(\Pi)$ is defined by requiring the normalized map 
$$
\mathcal{F}_{\Pi_f,\epsilon,[\Pi_{\infty}]}^0 := p^{\epsilon}(\Pi)^{-1}\mathcal{F}_{\Pi_f,\epsilon,[\Pi_{\infty}]}
$$
to be ${\rm Aut}({\mathbb C})$-equivariant, i.e., for all $\sigma \in {\rm Aut}({\mathbb C})$ one has
$$
\sigma \circ \mathcal{F}_{\Pi_f,\epsilon,[\Pi_{\infty}]}^0 = 
\mathcal{F}_{\Pi^{\sigma}_f,\epsilon,[\Pi^{\sigma}_{\infty}]}^0 \circ \sigma.
$$

\subsubsection{The cohomology classes}
\label{subsec:classes}
We now define the classes attached to the global Whittaker vectors $w_{\Pi_f}$ and $w_{\Sigma_f}$: 
\begin{equation}
\label{eqn:class-1}
\vartheta_{\Pi,\epsilon} := \mathcal{F}_{\Pi_f,\epsilon,[\Pi_{\infty}]}(w_{\Pi_f}), \ \ \ 
\vartheta^0_{\Pi,\epsilon} := \mathcal{F}_{\Pi_f,\epsilon,[\Pi_{\infty}]}^0(w_{\Pi_f}) 
= p^{\epsilon}(\Pi)^{-1}\vartheta_{\Pi,\epsilon},
\end{equation}
and similarly, 
\begin{equation}
\label{eqn:class-2}
\vartheta_{\Sigma,\eta} := \mathcal{F}_{\Sigma_f,\eta,[\Sigma_{\infty}]}(w_{\Sigma_f}), \ \ \ 
\vartheta^0_{\Sigma,\eta} := \mathcal{F}_{\Sigma_f,\eta,[\Sigma_{\infty}]}^0(w_{\Sigma_f}) 
= p^{\eta}(\Sigma)^{-1}\vartheta_{\Sigma,\eta}.
\end{equation}

Let $K_f$ be an open compact subgroup of $G_n({\mathbb A}_f)$
which fixes $w_{\Pi_f}$ and such that $\iota^*K_f$ fixes
$w_{\Sigma_f}$. Note that $\vartheta_{\Pi,\epsilon}$, which, by definition, lies in $H^{b_n}(\mathfrak{g}_{\infty}, K_{\infty}^0; V_{\Pi}\otimes M_{\mu}^{\vee})(\epsilon)$, actually lies in 
$H^{b_n}(\mathfrak{g}_{\infty}, K_{\infty}^0; V_{\Pi}^{K_f}\otimes M_{\mu}^{\vee})(\epsilon)$, and by the same token,
$\vartheta_{\Sigma,\eta} \in H^{b_{n-1}}(\mathfrak{g}_{\infty}, K_{\infty}^0; V_{\Sigma}^{\iota^*K_f}\otimes 
M_{\lambda}^{\vee})(\eta)$. Consider the manifolds:
\begin{eqnarray*}
S_n(K_f) & := & G_n({\mathbb Q})\backslash G_n({\mathbb A})/K_{n,\infty}^0K_f, \\
S_{n-1}(\iota^*K_f) & := & G_{n-1}({\mathbb Q})\backslash G_{n-1}({\mathbb A})/K_{n-1,\infty}^0\iota^*K_f
\end{eqnarray*}
Via certain standard isomorphisms (\cite[\S 3.3]{raghuram-shahidi-imrn}) we may identify the class $\vartheta_{\Pi,\epsilon}$ as a class 
in $H^{b_n}_{\rm cusp}(S_n(K_f), \mathcal{M}_{\mu}^{\vee})(\tilde{\Pi}_f)$ where $\tilde{\Pi_f} := \Pi_f \otimes \epsilon$ is a representation of $G_n({\mathbb A}_f) \otimes \pi_0(K_{n,\infty})$. 
Similarly, $\vartheta_{\Sigma,\eta} \in 
H^{b_{n-1}}_{\rm cusp}(S_{n-1}(\iota^*K_f), \mathcal{M}_{\lambda}^{\vee})(\tilde{\Sigma}_f)$. 

We recall that
cuspidal cohomology injects into cohomology with compact supports, i.e., $H^*_{\rm cusp} \hookrightarrow H^*_c$. 
(See \cite[p.129]{clozel}.) Hence $\vartheta_{\Pi,\epsilon}$ is a class in 
$H^{b_n}_c(S_n(K_f),\mathcal{M}_{\mu}^{\vee}),$ 
and similarly, $\vartheta_{\Sigma,\eta}$ lies  in $
H^{b_{n-1}}_c(S_{n-1}(\iota^*K_f),\mathcal{M}_{\lambda}^{\vee})$. (In this context, it is also helpful to bear in mind that cuspidal cohomology in fact injects into interior cohomology $H^*_! := {\rm Image}(H^*_c \to H^*)$. This is useful 
especially when dealing with rational structures; see \cite[Proof of Th\'eor\`eme 3.19]{clozel} or our paper \cite[\S 3.3]{raghuram-shahidi-imrn}.)

We remind the reader that the map $\iota: S_{n-1}(\iota^*K_f) \to S_n(K_f)$ is a proper map. Consider the pull back 
$\iota^*\vartheta_{\Pi,\epsilon}$ of $\vartheta_{\Pi,\epsilon}$ via $\iota$, which gives us a class in 
$H^{b_n}_c(S_{n-1}(\iota^*K_f), \iota^*\mathcal{M}_{\mu}^{\vee})$, where $\iota^*\mathcal{M}_{\mu}^{\vee}$ is the sheaf on $S_{n-1}(\iota^*K_f)$ attached to the restriction to $G_{n-1}$ of the 
representation $M_{\mu}^{\vee}$. We now define a certain pairing $\langle \vartheta_{\Sigma,\eta}, \iota^*\vartheta_{\Pi,\epsilon} \rangle_{\mathcal{C}(\iota^*K_f)}$, toward which we recall the construction of a cycle $\mathcal{C}(\iota^*K_f)$.

\subsubsection{The Harder-Mahnkopf cycle}
\label{sec:harder-mahnkopf-cycle}
We first explain the general principle of the construction. Let $M$ be a smooth connected orientable manifold of dimension $d$, $\overline{M}$ a compactification of $M$, and $\partial\overline{M}$ the boundary of $\overline{M}$.
Suppose that $M = \overline{M} - \partial\overline{M} = {\rm int}(\overline{M})$, and that
$M$ and $\overline{M}$ have the same homotopy type. 
(We should keep in mind the Borel-Serre compactification of a locally symmetric space.)
We have the following isomorphisms based on Poincar\'e duality: 
\begin{eqnarray*}
{\rm Hom}(H^d_c(M,{\mathbb Z}), {\mathbb Z}) & \simeq & {\rm Hom}(H_0(M,{\mathbb Z}), {\mathbb Z})
\simeq H^0(\overline{M},{\mathbb Z}) \\
& \simeq & H_d(\overline{M},\partial\overline{M}, {\mathbb Z}) \simeq {\mathbb Z} =: \langle [\vartheta_M] \rangle.
\end{eqnarray*}
To talk about $H^d_c(M,{\mathbb Z})$ we have transported the ${\mathbb Z}$-structure on singular cohomology
via the de Rham isomorphism. The fundamental class $[\vartheta_M]$ is well-defined up to a sign, and by the above isomorphisms, induces a 
functional $H^d_c(M,{\mathbb Z}) \to {\mathbb Z}$ which is nothing but integrating a compactly supported differential
form of degree $d$ over the entire manifold $M$ (with the chosen orientation, i.e., the choice of $[\vartheta_M]$). If the manifold $M$ is disconnected, but
has finitely many connected components, then in certain situations including the one we are interested in, it makes
sense to choose the fundamental classes for each connected component in a {\it consistent} manner. 

We digress a little to note that the above construction has good rationality properties. 
We recall (\cite[\S3]{raghuram-shahidi-imrn}) that by definition of the action of $\sigma \in {\rm Aut}({\mathbb C})$ on de Rham cohomology, as well as on cohomology with compact supports, one has 
$$
\sigma\left(\int_M \omega \right) = \int_M \omega^{\sigma}
$$
for any $\omega \in H^d_c(M,{\mathbb C})$, which may be re-written as $
\sigma \left(\langle [\vartheta_M], \omega \rangle \right) = 
\langle [\vartheta_M], \omega^{\sigma} \rangle $.

We now briefly review the Harder-Mahnkopf cycle, while referring the reader to \cite[5.1.1]{mahnkopf-jussieu} for 
all finer details. Recall that $K_{n,\infty}^1 := {\rm SO}(n) < G_{n,\infty}^0 = G_n({\mathbb R})^0$. For any open compact subgroup $K_f$ of $G_n({\mathbb A}_f)$ consider the manifold
$$
F_n(K_f) = G_n({\mathbb Q})\backslash G_n({\mathbb A})/K_{n,\infty}^1K_f.
$$
We let $d_n = n(n+1)/2 = {\rm dim}(F_n(K_f))$.
The connected components $F_{n,x}(K_f)$ of the manifold $F_n(K_f)$ are parametrized by $x \in {\mathbb Q}^*\backslash {\mathbb A}^{\times}/ {\mathbb R}_{>0}{\rm det}(K_f)$; indeed, for any such $x$, let 
$g_x \in G_n({\mathbb A}_f)$ be such that ${\rm det}(g_x)=x$, then 
$F_{n,x}(K_f)$ is identified with 
$\Gamma_x \backslash G_n({\mathbb R})^0 / K_{n,\infty}^1$ for the discrete subgroup 
$\Gamma_x = G_n({\mathbb Q}) \cap g_x K_f g_x^{-1}$ . 
These notations also apply to $G_{n-1}$ with any open compact subgroup $R_f$ of $G_{n-1}({\mathbb A}_f)$.
Choose an orientation on $X_{n-1} := G_{n-1}({\mathbb R})^0 / K_{n-1,\infty}^1$. Via the canonical map 
$X_{n-1} \to \Gamma_x \backslash X_{n-1} = F_{n-1,x}(K_f)$ we get a fundamental class 
$[\vartheta_{x,R_f}]$ on $F_{n-1,x}(R_f)$, i.e., $ [\vartheta_{x,R_f}]
\in H_{d_{n-1}}(\overline{F}_{n-1,x}(R_f), \partial\overline{F}_{n-1,x}(R_f),{\mathbb Z})$. 

At this point, it is convenient to work with ${\mathbb Q}$-coefficients. (Indeed, ultimately, 
it suffices to work with the ring obtained by inverting a finite set of primes determined by the 
primes where $\Pi$ and $\Sigma$ are ramified.) Now define 
$$
\mathcal{C}(R_f) = 
\frac{1}{{\rm vol}(R_f)}
\sum_{x \in {\mathbb Q}^*\backslash {\mathbb A}^{\times}/ {\mathbb R}_{>0}{\rm det}(R_f)}\ 
[\vartheta_{x,R_f}]
$$
which is the required cycle in 
$H_{d_{n-1}}(\overline{F}_{n-1}(R_f), \partial\overline{F}_{n-1}(R_f),{\mathbb Q})$.

Recall that $\pi_0(G_n)$ (resp., $\pi_0(K_{n,\infty})$) 
is the group of connected components of $G_n({\mathbb R})$ (resp., ${\rm O}(n)Z_n({\mathbb R})$).
We identify $\pi_0(K_{n,\infty}) \simeq \pi_0(G_n) \simeq {\mathbb Z}/2$. The nontrivial element may be taken to be represented by $\delta_n = {\rm diag}(-1,1,\dots,1)$. Right translations by $\delta_n$ on $G_n({\mathbb A})$, 
denoted $r_{\delta_n}$, induces an action of $\pi_0$ on $F_n(K_f)$, and by functoriality induces an action, denoted
$r_{\delta_n}^*$, on its (co-)homology groups. 
Applying these considerations to $G_{n-1}$, we get an action of $\pi_0(G_{n-1})$ on the cycle $\mathcal{C}(R_f)$
which is described in the following

\begin{lemma}
\label{lem:stable}
For any open compact subgroup $R_f$ of $G_{n-1}({\mathbb A}_f)$, the action of $\delta_{n-1}$ on the  
cycle $\mathcal{C}(R_f)$ is given by: 
$$
r_{\delta_{n-1}}^*\mathcal{C}(R_f) = (-1)^n\mathcal{C}(R_f).
$$
\end{lemma}

\begin{proof}
See Lemma 5.1.3 and the table in 5.2.2 of Mahnkopf \cite{mahnkopf-jussieu}. (This is a generalization of the 
fact that $\delta_2$ switches the (orientations on the) upper and lower half planes.)
\end{proof}

\subsubsection{The pairing 
$\langle \vartheta_{\Sigma,\eta}, \vartheta_{\Pi,\epsilon} \rangle_{\mathcal{C}(R_f)}$}
Now assume that $K_f$ is an open compact subgroup of $G_n({\mathbb A}_f)$, which for convenience may be taken to 
be a principal congruence subgroup of $G_n(\widehat{\mathbb Z})$. We let $R_f := \iota^*K_f$ which is an 
open compact subgroup of $G_{n-1}(\widehat{\mathbb Z})$. 
The map $\iota$ induces a proper map 
$\iota : F_{n-1}(R_f) \to S_n(K_f)$, which in turn induces a mapping 
$$
\iota^* : H^{\bullet}_c(S_n(K_f),\mathcal{M}_{\mu}^{\vee}) \to 
H^{\bullet}_c(F_{n-1}(R_f), \iota^*\mathcal{M}_{\mu}^{\vee}).
$$
Also, the canonical map $p : F_{n-1}(R_f) \to S_{n-1}(R_f)$ induces a mapping
$$
p^* : H^{\bullet}_c(S_{n-1}(R_f),\mathcal{M}_{\lambda}^{\vee}) \to 
H^{\bullet}_c(F_{n-1}(R_f), \mathcal{M}_{\lambda}^{\vee}).
$$ 
We invoke the hypothesis $\mu^{\vee} \succ \lambda$ which implies that $M_{\lambda}$ 
appears in $M_{\mu}^{\vee}|_{G_{n-1}} = \iota^*M_{\mu}^{\vee}$, and, in fact, it appears with multiplicity one. 
We fix a $G_{n-1}$-equivariant pairing, defined over ${\mathbb Q}$, and unique up to ${\mathbb Q}^*$, 
which we write as  
\begin{equation}
\label{eqn:pairing-M-mu-M-lambda}
\langle \cdot , \cdot \rangle : M_{\lambda}^{\vee} \times \iota^*M_{\mu}^{\vee} \to {\mathbb Q}
\end{equation}
and denote the corresponding morphism of sheaves as 
$
\langle \cdot , \cdot \rangle : \mathcal{M}_{\lambda}^{\vee} \otimes \iota^*\mathcal{M}_{\mu}^{\vee} \to 
\underline{\mathbb Q}.
$ 
Cup product together with the above pairing gives a map:
$$
\langle \cdot , \cdot \rangle \circ \cup : H^{b_{n-1}}_c(F_{n-1}(R_f), \mathcal{M}_{\lambda}^{\vee})
\times H^{b_n}_c(F_{n-1}(R_f), \iota^*\mathcal{M}_{\mu}^{\vee}) \to 
H^{d_{n-1}}_c(F_{n-1}(R_f), \underline{\mathbb Q}).
$$
This makes sense since $b_{n-1} + b_n = d_{n-1}$. We will abbreviate this map simply by $\cup$.
We digress for a moment to remind the reader that cupping cohomology classes, which makes sense in the context of singular cohomology, is the same as wedging cohomology classes, which makes sense in the context of de Rham cohomology. 
(See Griffiths-Harris \cite{griffiths-harris}.) To control rationality properties, it is best to think of the cup product, but to actually compute the pairing--as we will do later--it is best to think in terms of the wedge product.
This also permits us to write 
$$
p^*\vartheta_{\Sigma,\eta} \cup \iota^*\vartheta_{\Pi,\epsilon} = 
p^*\vartheta_{\Sigma,\eta} \wedge \iota^*\vartheta_{\Pi,\epsilon}.
$$
We now define the required pairing as 
\begin{equation}
\label{eqn:pairing}
\langle \vartheta_{\Sigma,\eta}, \vartheta_{\Pi,\epsilon} \rangle_{\mathcal{C}(R_f)} := 
\langle \mathcal{C}(R_f) , p^*\vartheta_{\Sigma,\eta} \cup \iota^*\vartheta_{\Pi,\epsilon} \rangle = 
\int_{\mathcal{C}(R_f)} p^*\vartheta_{\Sigma,\eta} \wedge \iota^*\vartheta_{\Pi,\epsilon}
\end{equation}
where the second equality is given by Poincar\'e duality as described in \ref{sec:harder-mahnkopf-cycle}. 
(See also \cite[Diagram (5.3)]{mahnkopf-jussieu}.)

\subsubsection{The pairing at infinity and a nonvanishing hypothesis}
\label{subsec:nonvanishing}

We recall again that the class $\vartheta_{\Pi,\epsilon}$ is the image of a  
certain global finite Whittaker vector $w_{\Pi_f}$ under the map 
$\mathcal{F}_{\Pi_f,\epsilon,[\Pi_{\infty}]}$. (All these comments also apply to $\vartheta_{\Sigma,\eta}$.)  
We recall \cite[\S 3.3]{raghuram-shahidi-imrn} that this map is the composition of
the three isomorphisms: 
\begin{eqnarray*}
W(\Pi_f) & \longrightarrow & 
W(\Pi_f) \otimes 
H^{b_n}(\mathfrak{g}_{n,\infty},K_{n,\infty}^0; W(\Pi_{\infty}) \otimes 
M_{\mu}^{\vee})(\epsilon) \\
& \longrightarrow & 
H^{b_n}(\mathfrak{g}_{n,\infty},K_{n,\infty}^0; W(\Pi) \otimes 
M_{\mu}^{\vee})(\epsilon) \\
& \longrightarrow & 
H^{b_n}(\mathfrak{g}_{n,\infty},K_{n,\infty}^0; V_{\Pi} \otimes M_{\mu}^{\vee})(\epsilon),
\end{eqnarray*}
where the first map is $w_f \mapsto w_f \otimes [\Pi_{\infty}]$; 
the second map is the obvious one; and the third map is the map induced in cohomology 
by the inverse of the map which gives the Fourier coefficient of a cusp form in $V_{\Pi}$--the space
of functions in $\mathcal{A}_{{\rm cusp}}(G({\mathbb Q}) \backslash G({\mathbb A}))$ which 
realizes $\Pi$. In particular, in computing the pairing 
$\langle \vartheta_{\Sigma,\eta}, \iota^*\vartheta_{\Pi,\epsilon} \rangle_{\mathcal{C}(R_f)}$, we will
be computing a pairing at infinity, and a pairing with the finite vectors $w_{\Pi_f}$ and $w_{\Sigma_f}$. 
The latter is indeed the Rankin--Selberg integral (at $s=1/2$) appearing in the left hand side of 
Proposition~\ref{prop:rankin-selberg}. We now discuss the pairing at infinity.

To compute the pairing at infinity, we follow the argument in \cite[\S5.1.4]{mahnkopf-jussieu}.
Fix a basis $\{{\bf x}_i\}$ for $(\mathfrak{g}_{n,\infty}/\mathfrak{k}_{n,\infty})^*$, and a basis 
$\{{\bf y}_j\}$ for $(\mathfrak{g}_{n-1,\infty}/\mathfrak{k}_{n-1,\infty})^*$, such that 
$\iota^*{\bf x}_j = {\bf y}_j$ for all 
$1 \leq j \leq {\rm dim}(\mathfrak{g}_{n-1,\infty}/\mathfrak{k}_{n-1,\infty})^* = {\rm dim}(X_{n-1}) = d_{n-1}$, and 
$\iota^*{\bf x}_i = 0$ if $i > d_{n-1}$. We further note that 
${\bf y}_1\wedge {\bf y}_2 \wedge \cdots \wedge {\bf y}_{d_{n-1}}$ corresponds to 
a $G_{n-1}({\mathbb R})^0$-invariant measure on $X_{n-1}$. 
Let $\{m_{\alpha}\}$ (resp., $\{m_{\beta}\}$) be a ${\mathbb Q}$-basis for $M_{\mu}^{\vee}$ (resp., 
$M_{\lambda}^{\vee}$), and recall that we have a pairing $\langle \cdot , \cdot \rangle$ between these 
modules as in (\ref{eqn:pairing-M-mu-M-lambda}). Now the class 
$[\Pi_{\infty}]$ is represented by a $K_{n,\infty}^0$-invariant element in 
$\wedge^{b_n}(\mathfrak{g}_{n,\infty}/\mathfrak{k}_{n,\infty})^* \otimes W(\Pi_{\infty}) \otimes M_{\mu}^{\vee}$
which we write as 
\begin{equation}
\label{eqn:pi-infty}
[\Pi_{\infty}] = 
\sum_{{\bf i} = i_1 < \cdots < i_{b_n}}
\sum_{\alpha} {\bf x}_{\bf i} \otimes w_{\infty,\bf{i},\alpha}  \otimes m_{\alpha},
\end{equation}
where $w_{\infty,\bf{i},\alpha} \in W(\Pi_{\infty},\psi_{\infty})$, and 
similarly, $[\Sigma_{\infty}]$ is represented by a $K_{n-1,\infty}^0$-invariant element in 
$\wedge^{b_{n-1}}(\mathfrak{g}_{n-1,\infty}/\mathfrak{k}_{n-1,\infty})^* \otimes W(\Sigma_{\infty}) \otimes 
M_{\lambda}^{\vee}$ which we write as:
\begin{equation}
\label{eqn:sigma-infty}
[\Sigma_{\infty}] = 
\sum_{{\bf j} = j_1 < \cdots < j_{b_{n-1}}}
\sum_{\beta} {\bf y}_{\bf j} \otimes w_{\infty,\bf{j},\beta}  \otimes m_{\beta},
\end{equation}
with $w_{\infty,\bf{j},\beta} \in W(\Sigma_{\infty},\psi^{-1}_{\infty})$.
We now define a pairing at infinity by
\begin{equation}
\label{eqn:pairing-infinity}
\langle [\Pi_{\infty}], [\Sigma_{\infty}] \rangle = 
\sum_{{\bf i}, {\bf j}} s({\bf i}, {\bf j}) 
\sum_{\alpha, \beta} 
\langle m_{\beta} , m_{\alpha} \rangle
\Psi_{\infty}(1/2, w_{\infty,{\bf i},\alpha}, w_{\infty,{\bf j},\beta})
\end{equation}
where $s({\bf i},{\bf j}) \in \{0,-1,1\}$ is defined by 
$\iota^*{\bf x}_{\bf i} \wedge {\bf y}_{\bf j} = s({\bf i},{\bf j}) {\bf y}_1 \wedge {\bf y}_2 \wedge \cdots 
\wedge {\bf y}_{d_{n-1}}$. Recall that 
$\Psi_{\infty}(1/2, w_{\infty,{\bf i},\alpha}, w_{\infty,{\bf j},\beta})$ is defined only after 
meromorphic continuation; see Cogdell--Piatetskii-Shapiro \cite[Theorem 1.2]{cogdell-ps}. Note that the assumption
`$s = 1/2$ is critical' ensures that the integrals $\Psi_{\infty}(1/2, w_{\infty,{\bf i},\alpha}, w_{\infty,{\bf j},\beta})$ are all finite, hence $\langle [\Pi_{\infty}], [\Sigma_{\infty}] \rangle$ is finite. 
We now make the following nonvanishing hypothesis about this pairing at infinity: 
\begin{hypo}
\label{hypo:nonvanishing}
$
\langle [\Pi_{\infty}], [\Sigma_{\infty}] \rangle \neq 0.
$
\end{hypo}

This nonvanishing hypothesis is currently a limitation of this technique. 
It has shown up in 
several other works based on the same, or at any rate similar, techniques. 
See for instance Ash-Ginzburg \cite{ash-ginzburg}, 
Harris \cite{harris},
Kasten-Schmidt \cite{kasten-schmidt}, 
Kazhdan-Mazur-Schmidt \cite{kazhdan-mazur-schmidt},
Mahnkopf \cite{mahnkopf-jussieu},   
and Schmidt \cite{schmidt}. It is widely hoped that this assumption is valid; for example, 
Mahnkopf \cite[\S6]{mahnkopf-jussieu}
proves a necessary condition for this nonvanishing assumption,  
Schmidt \cite{schmidt} proved it for $n=3$ in the case of trivial coefficients ($\mu=0$ and $\lambda=0$), 
and Kasten-Schmidt \cite[\S4]{kasten-schmidt} have recently proved it for $n=3$ for nontrivial coefficients.
{\it It is an important technical problem to be able to prove this nonvanishing 
hypothesis.} 

For the rest of this paper we assume that Hypothesis~\ref{hypo:nonvanishing} is valid. Observe that 
the quantity $\langle [\Pi_{\infty}], [\Sigma_{\infty}] \rangle $ depends only on the weights $\mu$ and 
$\lambda$, since the weight $\mu$ determines the infinitesimal character of $M_{\mu}^{\vee}$ which 
in turn determines $\Pi_{\infty}$ (\cite[\S5.1]{raghuram-shahidi-imrn}), and similarly, since  
$\lambda$ determines $\Sigma_{\infty}$. We now define,
what may loosely be called as the period at infinity, a nonzero complex number $p_{\infty}(\mu,\lambda)$ 
given by:
\begin{equation}
\label{eqn:p-infinity}
p_{\infty}(\mu,\lambda) := \frac{1}{\langle [\Pi_{\infty}], [\Sigma_{\infty}] \rangle}.
\end{equation}
Ultimately, if one is able to explicitly compute everything, then one should expect $p_{\infty}(\mu,\lambda)$
to be a power of $(2\pi i)$.

\subsubsection{The main identity for the central critical value of Rankin--Selberg $L$-functions}

\begin{thm}[Main Identity]
\label{thm:main-identity}
Let $\Pi$ be a regular algebraic cuspidal automorphic representation of ${\rm GL}_n({\mathbb A}_{\mathbb Q})$, and 
let $\Sigma$  be a regular algebraic cuspidal automorphic representation of 
${\rm GL}_{n-1}({\mathbb A}_{\mathbb Q})$.  Let $\mu \in X^+_0(T_n)$ be such that $\Pi \in {\rm Coh}(G_n,\mu^{\vee})$, and let $\lambda \in X^+_0(T_{n-1})$ be such that $\Sigma \in {\rm Coh}(G_{n-1},\lambda^{\vee})$. Assume that 
$\mu^{\vee}  \succ \lambda$, and that $s=1/2$ is critical for 
$L_f(s, \Pi \times \Sigma)$. We attach a canonical pair of 
signs $\epsilon, \eta \in \{\pm\}$ to the pair $(\Pi,\Sigma)$ as follows: 
\begin{enumerate}
\item $\epsilon = (-1)^n\eta$. 
\item 
   \begin{itemize}
     \item If $n$ is odd then let $\epsilon = \omega_{\Pi_{\infty}}(-1)(-1)^{{\rm wt}(\mu)/2}$; 
     \item if $n$ is even then let $\eta = \omega_{\Sigma_{\infty}}(-1)(-1)^{{\rm wt}(\lambda)/2}.$ 
   \end{itemize}
\end{enumerate} 
Let $w_{\Pi_f}$ and $w_{\Sigma_f}$ be the Whittaker vectors defined in \ref{subsec:whittaker}.
We let $K_f$ be any open compact subgroup of $G_n({\mathbb A}_f)$ which fixes $w_{\Pi_f}$ and such that
$R_f := \iota^*K_f$ fixes $w_{\Sigma_f}$. We let $\vartheta^0_{\Pi,\epsilon}$ and 
$\vartheta^0_{\Sigma,\eta}$ be the normalized classes defined in (\ref{eqn:class-1}) and (\ref{eqn:class-2}).
There exists nonzero complex numbers $p^{\epsilon}(\Pi)$ and 
$p^{\eta}(\Sigma)$ as in \ref{subsec:periods}, and assuming the validity of 
Hypothesis~\ref{hypo:nonvanishing} there is 
a nonzero complex number $p_{\infty}(\mu,\lambda)$ as in \ref{subsec:nonvanishing} such that
$$
\frac{L_f(1/2, \Pi \times \Sigma)}
{p^{\epsilon}(\Pi)\, p^{\eta}(\Sigma)\, p_{\infty}(\mu,\lambda)} \ = \  
\frac{\prod_{v \in S_{\Sigma}} L(1/2, \Pi_v \times \Sigma_v)}
{{\rm vol}(\Sigma)\, \prod_{v \notin S_{\Sigma} \cup \{\infty \} } c_{\Pi_v}}\, 
\langle \vartheta^0_{\Sigma,\eta}, \vartheta^0_{\Pi,\epsilon}\rangle_{\mathcal{C}(R_f)},
$$
where the pairing on the right hand side is defined in (\ref{eqn:pairing}), the nonzero rational number  
${\rm vol}(\Sigma)$ is as in Proposition~\ref{prop:rankin-selberg}, 
and $c_{\Pi_v}$ is defined in \ref{sec:newvectors}.
\end{thm}

\begin{proof}
By definition of the normalization of the cohomology classes, it suffices to prove 
$$
\frac{L_f(1/2, \Pi \times \Sigma)}
{p_{\infty}(\mu,\lambda)} \ = \  
\frac{\prod_{v \in S_{\Sigma}} L(1/2, \Pi_v \times \Sigma_v)}
{{\rm vol}(\Sigma)\, \prod_{v \notin S_{\Sigma} \cup \{\infty \}} c_{\Pi_v}}\, 
\langle \vartheta_{\Sigma,\eta}, \vartheta_{\Pi,\epsilon}\rangle_{\mathcal{C}(R_f)}.
$$
By definition of the pairing at infinity, it suffices then to verify 
\begin{equation}
\label{eqn:verify-pairing}
\langle \vartheta_{\Sigma,\eta}, \vartheta_{\Pi,\epsilon} \rangle_{\mathcal{C}(R_f)} = 
\frac{{\rm vol}(\Sigma)\, \prod_{v \notin S_{\Sigma} \cup \{\infty \}} c_{\Pi_v}}
{\prod_{v \in S_{\Sigma}} L(1/2, \Pi_v \times \Sigma_v)}\,
\langle [\Pi_{\infty}], [\Sigma_{\infty}] \rangle \, L_f(1/2, \Pi \times \Sigma).
\end{equation}

The class $\vartheta_{\Pi,\epsilon} \in H^{b_n}(\mathfrak{g}_{n,\infty}, K_{n,\infty}^0; V_{\Pi}\otimes M_{\mu}^{\vee})(\epsilon)$, as in \ref{subsec:nonvanishing}, is represented by a $K_{n,\infty}^0$-invariant element in 
$\wedge^{b_n}(\mathfrak{g}_{n,\infty}/\mathfrak{k}_{n,\infty})^* \otimes V_{\Pi} \otimes M_{\mu}^{\vee}$ which 
we write as
$$
\vartheta_{\Pi,\epsilon} = 
\sum_{{\bf i} = i_1 < \cdots < i_{b_n}}
\sum_{\alpha} {\bf x}_{\bf i} \otimes \phi_{\bf{i},\alpha}  \otimes m_{\alpha}. 
$$
Similarly, we write $\vartheta_{\Sigma, \eta}$ as 
$$
\vartheta_{\Sigma, \eta} = 
\sum_{{\bf j} = j_1 < \cdots < j_{b_{n-1}}}
\sum_{\beta} {\bf y}_{\bf j} \otimes \phi_{\bf{j},\beta}  \otimes m_{\beta}, 
$$
representing a $K_{n-1,\infty}^0$-invariant element in 
$\wedge^{b_{n-1}}(\mathfrak{g}_{n-1,\infty}/\mathfrak{k}_{n-1,\infty})^* \otimes V_{\Sigma} \otimes 
M_{\lambda}^{\vee}$. 
Let $w_{\bf{i},\alpha}$ be the Whittaker vector in $W(\Pi,\psi)$ corresponding to $\phi_{\bf{i},\alpha}$, and 
similarly, $w_{\bf{j},\beta}$ be the Whittaker vector in $W(\Sigma,\psi^{-1})$ corresponding to $\phi_{\bf{j},\beta}$.
Unravelling the definitions, we have the decompositions
$$
w_{\bf{i},\alpha} = w_{\infty,\bf{i},\alpha} \otimes w_{\Pi_f}, \ \ {\rm and}\ \ 
w_{\bf{j},\beta} = w_{\infty,\bf{j},\beta} \otimes w_{\Sigma_f},
$$
where the vectors at infinity are exactly as in (\ref{eqn:pi-infty}) and (\ref{eqn:sigma-infty}). 

To verify 
(\ref{eqn:verify-pairing}) we begin with the definition of the pairing
$$
\langle \vartheta_{\Sigma,\eta}, \vartheta_{\Pi,\epsilon} \rangle_{\mathcal{C}(R_f)} = 
\int_{\mathcal{C}(R_f)} p^*\vartheta_{\Sigma,\eta} \wedge \iota^*\vartheta_{\Pi,\epsilon},  
$$
and observe that the integral on the right hand side is stable under the action of $\pi_0$ exactly 
when $\epsilon\eta = (-1)^n$; this may be seen by using Lemma~\ref{lem:stable} 
(just as in \cite[5.2.2]{mahnkopf-jussieu}).
Next, we note that 
the right hand side may be written as
$$
\frac{1}{{\rm vol}(R_f)}
\sum_{{\bf i},{\bf j},\alpha,\beta} s({\bf i},{\bf j})\langle m_{\beta},m_{\alpha}\rangle
\int_{ G_{n-1}({\mathbb Q}) \backslash G_{n-1}({\mathbb A}) / K_{n-1,\infty}^1R_f }
\phi_{{\bf i},\alpha}(\iota(g))\phi_{{\bf j},\beta}(g)\, dg.
$$
By the choice of the measure $dg$, this simplifies to 
$$
\sum_{{\bf i},{\bf j},\alpha,\beta} s({\bf i},{\bf j})\langle m_{\beta},m_{\alpha}\rangle
\int_{ G_{n-1}({\mathbb Q}) \backslash G_{n-1}({\mathbb A})}
\phi_{{\bf i},\alpha}(\iota(g))\phi_{{\bf j},\beta}(g)\, dg. 
$$
The inner integral is nothing but $I(1/2,\phi_{{\bf i},\alpha}, \phi_{{\bf j},\beta})$. Applying 
Proposition~\ref{prop:rankin-selberg} we get 
$$
\frac{{\rm vol}(\Sigma)\, \prod_{v \notin S_{\Sigma} \cup \{\infty \}} c_{\Pi_v}}
{\prod_{v \in S_{\Sigma}} L(1/2, \Pi_v \times \Sigma_v)}\,
L_f(1/2, \Pi \times \Sigma) 
\sum_{{\bf i},{\bf j},\alpha,\beta} s({\bf i},{\bf j})\langle m_{\beta},m_{\alpha}\rangle 
\Psi_{\infty}(1/2, w_{\infty,{\bf i},\alpha}, w_{\infty,{\bf j},\beta}). 
$$
The proof follows from the definition of $\langle [\Pi_{\infty}], [\Sigma_{\infty}] \rangle$.
\end{proof}

\subsection{Proof of Theorem~\ref{thm:rankin-selberg}}

The proof follows by applying $\sigma \in {\rm Aut}({\mathbb C})$ to the main identity in 
Theorem~\ref{thm:main-identity}. We now would like to know the Galois equivariance of all 
the quantities on the right hand side of the main identity. This we delineate in the following propositions: 

\begin{prop}
\label{prop:galois-poincare}
Let $\varpi \in H^{b_n}_c(S_n(K_f),\mathcal{M}_{\mu^{\vee}})$, and 
$\varsigma \in H^{b_{n-1}}_c(S_{n-1}(\iota^*K_f),\mathcal{M}_{\lambda^{\vee}})$.  
For any $\sigma \in {\rm Aut}({\mathbb C})$ we have 
$$
\sigma \left(\langle \varsigma \, , \, \varpi \rangle_{\mathcal{C}(R_f)} \right) = 
\langle {}^{\sigma}\!\varsigma \, , \, {}^{\sigma}\!\varpi \rangle_{\mathcal{C}(R_f)}.
$$
\end{prop}

\begin{proof}
This follows from the well-known Galois equivariance property of Poincar\'e duality 
(see, for example, Mahnkopf \cite[proof of Lemma 1.2]{mahnkopf-crelle}), coupled with the fact that
the maps $\iota^*$ and $p^*$ are Galois equivariant. 
\end{proof}

\begin{prop}
\label{prop:galois-classes}
The classes $\vartheta^0_{\Pi,\epsilon}$ and $\vartheta^0_{\Sigma,\eta}$, constructed in 
\ref{subsec:classes}, have the following behaviour under $\sigma \in {\rm Aut}({\mathbb C})$:
$$
{}^{\sigma}\!\vartheta^0_{\Pi, \epsilon}  =  \sigma(\omega_{\Sigma_f}(t_{\sigma}))\vartheta^0_{\Pi^{\sigma},\epsilon}, 
\ \ \ 
{}^{\sigma}\!\vartheta^0_{\Sigma, \eta} =  \vartheta^0_{\Sigma^{\sigma},\eta}.
$$
\end{prop}

\begin{proof}
By definition of the classes, and the Galois equivariance of $\mathcal{F}^0$, we have 
$$
{}^{\sigma}\!\vartheta^0_{\Pi, \epsilon}  = {}^{\sigma}\!\mathcal{F}_{\Pi_f,\epsilon,[\Pi_{\infty}]}^0(w_{\Pi_f})
= \mathcal{F}_{\Pi^{\sigma}_f,\epsilon,[\Pi_{\infty}]}^0({}^{\sigma}\!w_{\Pi_f}).
$$
Next, we note that by the choice of the vector $w_{\Pi_f}$, we have
$$
{}^{\sigma}\!w_{\Pi_f} = {}^{\sigma}\!(\otimes_{v\notin S_{\Sigma}}w_{\Pi_v}^0 \otimes_{v \in S_{\Sigma}} w_{\Pi,v}) = 
\otimes_{v\notin S_{\Sigma}} {}^{\sigma}\!w_{\Pi_v}^0 \otimes_{v \in S_{\Sigma}} {}^{\sigma}\!w_{\Pi,v},
$$
the second equality is due to the compatibility of local and global actions of $\sigma$ .
For $v \notin S_{\Sigma}$ we know that ${}^{\sigma}\!w_{\Pi_v}^0 = w_{\Pi^{\sigma}_v}^0$. However, for 
$v \in S_{\Sigma}$, we note first that the support of ${}^{\sigma}\!w_{\Pi,v}$ restricted to $G_{n-1}$ is also 
the same double coset $N_{n-1}({\mathbb Q}_v)t_{\Sigma_v}K_{n-1}(\mathfrak{f}(\Sigma_v))$, and on this double 
coset it is given by 
\begin{eqnarray*}
{}^{\sigma}\!w_{\Pi,v}|_{G_{n-1}}(ut_{\Sigma_v}k) & = &
\sigma\left(w_{\Pi,v}|_{G_{n-1}}\left(\left(\begin{array}{cccc} 
t_{\sigma,v}^{-(n-1)} & & & \\
& t_{\sigma,v}^{-(n-2)} & & \\
& & \ldots & \\
& & & t_{\sigma,v}^{-1}\end{array}\right)ut_{\Sigma_v}k\right)\right) \\
& = &  \sigma(\psi_v(t_{\sigma}^{-1}u)\omega_{\Sigma_v}^{-1}(t_{\sigma,v}^{-1}k_{n-1,n-1})) \\
& = & \sigma(\omega_{\Sigma_v}(t_{\sigma,v})) w_{\Pi^{\sigma},v}|_{G_{n-1}}(ut_{\Sigma_v}k).
\end{eqnarray*}
Hence, for $v \in S_{\Sigma}$ we have 
${}^{\sigma}\!w_{\Pi,v} = \sigma(\omega_{\Sigma_v}(t_{\sigma,v}))w_{\Pi_v^{\sigma}}$. Noting that 
$\otimes_{v \in S_{\Sigma}} \omega_{\Sigma_v}(t_{\sigma,v}) = \omega_{\Sigma_f}(t_{\sigma})$, we get 
${}^{\sigma}\!w_{\Pi_f} = \sigma(\omega_{\Sigma_f}(t_{\sigma})) w_{\Pi^{\sigma}_f}$, which finishes the 
proof of the first assertion of the proposition. 
The proof of Galois equivariance of the class $\vartheta^0_{\Sigma, \eta}$ is similar (and simpler). 
\end{proof}

For later use we note that the above variance for $\vartheta^0_{\Pi,\epsilon}$ may also be stated in terms 
of Gauss sums. 

\begin{cor}
\label{cor:gauss-sum}
For any $\sigma \in {\rm Aut}({\mathbb C})$ we have 
$$
{}^{\sigma}\!\vartheta^0_{\Pi, \epsilon}  =  
\frac{\sigma(\mathcal{G}(\omega_{\Sigma_f}))}
{\mathcal{G}(\omega_{\Sigma_f^{\sigma}})}
\vartheta^0_{\Pi^{\sigma},\epsilon}
$$
\end{cor}

\begin{proof}
Follows from the above proposition and (\ref{eqn:variation1}).
\end{proof}

\begin{prop}
\label{prop:galois-localvalues}
For a finite place $v$ of ${\mathbb Q}$, let $\Pi_v$ and $\Sigma_v$ be (any) irreducible 
admissible  representations of $G_n({\mathbb Q}_v)$ and $G_{n-1}({\mathbb Q}_v)$. Then
$$
\sigma(L(1/2, \Pi_v \times \Sigma_v)) = L(1/2, \Pi_v^{\sigma} \times \Sigma_v^{\sigma}).
$$
\end{prop}

\begin{proof}
Let $F = {\mathbb Q}_v$, or for that matter, any non-archimedean local field with its associated baggage of notations
like $\mathcal{O}$, $\mathcal{P}$, $q$, etc. Let $\pi$ be any irreducible admissible representation of ${\rm GL}_m(F)$.
From Clozel \cite[Lemma 4.6]{clozel} we have 
$$
{}^{\sigma}\!L\left(s+\frac{1-m}{2},\pi\right) = L\left(s+\frac{1-m}{2},\pi^{\sigma}\right).
$$ 
In the left hand side, if $L(s+(1-m)/2,\pi) = P(q^{-s})^{-1}$ for a polynomial $P(X) \in {\mathbb C}[X]$ with $P(0)=1$, then ${}^{\sigma}\!P(q^{-s})$ is obtained by applying $\sigma$ to the coefficients of $P(X)$. Now assume that $m$ is even. Then 
\begin{eqnarray*}
\sigma(L(1/2,\pi)) & = & \sigma\left(L\left(s+\frac{1-m}{2},\pi\right)|_{s=m/2}\right) \\
& = & \sigma(P(q^{-m/2})^{-1}) \\
& = & {}^{\sigma}\!P(q^{-m/2})^{-1} \ \ \ \ \mbox{(since $m$ is even)}\\
& = & {}^{\sigma}\!L(1/2,\pi).
\end{eqnarray*}
From the above mentioned lemma we have
\begin{equation}
\label{eqn:clozel}
\sigma(L(1/2,\pi)) = L(1/2,\pi^{\sigma}).
\end{equation}

We need a result of Henniart about the local Langlands correspondence for ${\rm GL}_m(F)$. 
We denote this correspondence as $\pi \mapsto \tau(\pi)$ and $\tau \mapsto \pi(\tau)$ between irreducible admissible 
representations $\pi$ of ${\rm GL}_m(F)$ and $m$-dimensional semisimple representations $\tau$ of the Weil-Deligne group
$W_F' = W_F \times {\rm SL}_2({\mathbb C})$. For any $\sigma \in {\rm Aut}({\mathbb C})$, we let $\epsilon_{\sigma}$ 
denote the quadratic character $x \mapsto \sigma(|x|^{1/2})/|x|^{1/2}$ of $F^*$. From Henniart \cite[7.4]{henniart} we have 
\begin{equation}
\label{eqn:henniart}
\pi(\tau)^{\sigma} = \pi(\tau^{\sigma})\epsilon_{\sigma}^{m-1}, \ \mbox{and} \ 
\tau(\pi)^{\sigma} = \tau(\pi^{\sigma})\epsilon_{\sigma}^{m-1}.
\end{equation}
In \cite{henniart} the Langlands correspondence is stated between the Grothendieck group 
generated by irreducible representations of the Weil group $W_F$ on the one hand and the Grothendieck group
generated by irreducible supercuspidal representations on the other. In particular, (\ref{eqn:henniart}) is stated for such representations. However, one can easily see that (\ref{eqn:henniart}) remains true as we have 
stated it if one defines the action of $\sigma \in {\rm Aut}({\mathbb C})$ on semisimple representations of $W_F'$ by 
$$
(\tau_1\otimes\mathcal{S}_{m_1} \oplus \cdots \oplus \tau_r\otimes\mathcal{S}_{m_r})^{\sigma} = 
\tau_1^{\sigma}\otimes\mathcal{S}_{m_1} \oplus \cdots \oplus \tau_1^{\sigma}\otimes\mathcal{S}_{m_1}
$$
for irreducible representations $\tau_i$ of $W_F$, and integers $m_i$, 
where for any integer $k \geq 1$ the $k$-dimensional
irreducible representation of ${\rm SL}_2({\mathbb C})$ is denoted $\mathcal{S}_k$.

Now, let $\pi_1$ and $\pi_2$ be irreducible admissible representations of ${\rm GL}_{m_1}(F)$ and 
${\rm GL}_{m_2}(F)$, respectively. Define the `automorphic tensor product' by 
$\pi_1\boxtimes\pi_2 := \pi(\tau(\pi_1)\otimes\tau(\pi_2))$. One can check from (\ref{eqn:henniart}) that 
for any $\sigma \in {\rm Aut}({\mathbb C})$ we have 
\begin{equation}
\label{eqn:box}
(\pi_1\boxtimes\pi_2)^{\sigma} = (\pi_1^{\sigma} \boxtimes \pi_2^{\sigma}) 
\otimes \epsilon_{\sigma}^{(1-m_1)(1-m_2)}.
\end{equation}

The proposition follows from (\ref{eqn:clozel}) and (\ref{eqn:box}) by taking $m_1 = n$, $m_2=n-1$, 
$\pi_1 = \Pi_v$, $\pi_2 = \Sigma_v$, and $\pi = \pi_1\boxtimes\pi_2$, while keeping in mind that 
$L(s,\pi_1 \times \pi_2)  = L(s,\pi_1\boxtimes\pi_2)$. 
\end{proof}

Albeit the above proposition is not hard to prove, we wish to emphasize the fact that it is a crucial ingredient
in our paper. The moral being that the possibly transcendental parts of special values of $L$-functions are 
already captured by partial $L$-functions, i.e., we can ignore finitely many places as these local $L$-values 
are in the rationality field.

\begin{prop}
\label{prop:galois-c}
$$
\sigma(c_{\Pi_v}) = c_{\Pi^{\sigma}_v}.
$$
\end{prop}

\begin{proof}
See Mahnkopf \cite[p.621]{mahnkopf-jussieu}, where it is mentioned that the proof is the same argument as in the proof of \cite[Proposition 2.3(c)]{mahnkopf-crelle}.
\end{proof}

\begin{proof}[Proof of Theorem~\ref{thm:rankin-selberg}] 
Apply $\sigma \in {\rm Aut}({\mathbb C})$ to the main identity in  
Theorem~\ref{thm:main-identity} to get 
$$
\sigma\left(\frac{L_f(1/2,\Pi \times \Sigma)}
{p^{\epsilon}(\Pi)p^{\eta}(\Sigma)p_{\infty}(\mu,\lambda)}\right) 
=  \sigma\left(\frac{\prod_{v \in S_{\Sigma}} L(1/2, \Pi_v \times \Sigma_v)}
{{\rm vol}(\Sigma)\, \prod_{v \notin S_{\Sigma}} c_{\Pi_v}}\, 
\langle \vartheta^0_{\Sigma,\eta}, \vartheta^0_{\Pi,\epsilon}\rangle_{\mathcal{C}(R_f)}\right).
$$
Applying Propositions~\ref{prop:galois-poincare}, 
\ref{prop:galois-localvalues}, \ref{prop:galois-c}, and Corollary~\ref{cor:gauss-sum} to the right hand side we have
$$
\frac{\prod_{v \in S_{\Sigma^{\sigma}}} L(1/2, \Pi_v^{\sigma} \times \Sigma_v^{\sigma})}
{{\rm vol}(\Sigma^{\sigma})\, \prod_{v \notin S_{\Sigma^{\sigma}}} c_{\Pi_v^{\sigma}}}\, 
\frac{\sigma(\mathcal{G}(\omega_{\Sigma_f}))}
{\mathcal{G}(\omega_{\Sigma_f^{\sigma}})}
\langle \vartheta^0_{\Sigma^{\sigma},\eta}, \vartheta^0_{\Pi^{\sigma},\epsilon}\rangle_{\mathcal{C}(R_f)}  
 = \frac{\sigma(\mathcal{G}(\omega_{\Sigma_f}))}
{\mathcal{G}(\omega_{\Sigma_f^{\sigma}})}
\frac{L_f(1/2,\Pi^{\sigma} \times \Sigma^{\sigma})}
{p^{\epsilon}(\Pi^{\sigma})p^{\eta}(\Sigma^{\sigma})p_{\infty}(\mu,\lambda)}
$$
from which the theorem follows. 
\end{proof}

\subsection{The effect of changing rational structures}
\label{sec:different-rational-structures}
In this section we study the effect of changing rational structures involved in the definition of the 
periods. Recall, from \cite{raghuram-shahidi-imrn}, that the period $p^{\epsilon}(\Pi)$ is defined by comparing a rational structure on the Whittaker model $W(\Pi_f)$ with that on a suitable cohomology space, namely, 
$H^{b_n}(\mathfrak{g}_{\infty},K_{\infty}^0, V_{\Pi}\otimes M_{\mu}^{\vee})(\epsilon)$. The rational structure 
on this cohomology space comes ultimately from a canonical ${\mathbb Z}$-structure on singular cohomology, however,
the rational structure on the Whittaker model is not so canonical. In this section we draw attention to some other 
(very natural looking) rational structures on $W(\Pi_f)$. It should be borne in mind that a rational structure 
on $W(\Pi_f)$ is unique up to homotheties, so there is indeed an emphasis on the ``naturality" of the definition.

For each $r \in {\mathbb Z}$, we define an action of ${\rm Aut}({\mathbb C})$ on $W(\Pi_f)$ as follows: 
For $\sigma \in {\rm Aut}({\mathbb C})$, $w \in W(\Pi_f, \psi)$,  define 
$$
\sigma_r(w)(g) = \sigma\left(w\left(\left(\begin{array}{cccc}
t_{\sigma}^{r-(n-1)} & & & \\
& t_{\sigma}^{r-(n-2)} & & \\
& & \ldots & \\
& & & t_{\sigma}^r \end{array}\right)g\right)\right)
$$
for all $g \in G_n({\mathbb A}_f)$. It is easy to see that $w \mapsto \sigma_r(w)$ is a $G_n({\mathbb A}_f)$-equivariant, $\sigma$-linear isomorphism from $W(\Pi,\psi)$ onto $W(\Pi^{\sigma},\psi)$. For $r=0$ this is nothing
but the previous action we had considered. We can relate the two actions by pulling out a central character:
\begin{equation}
\label{eqn:r-action} 
\sigma_r(w) = 
\sigma(\omega_{\Pi}(t_{\sigma}^r))\sigma_0(w) = 
\left(\frac{\sigma(\mathcal{G}(\omega_{\Pi_f}))}
{\mathcal{G}(\omega_{\Pi_f^{\sigma}})}\right)^r 
\sigma_0(w).
\end{equation}
If $w_0 \in W(\Pi)$ is the normalized new vector that is fixed by $\sigma_0$, for all $\sigma \in 
{\rm Gal}({\mathbb C}/{\mathbb Q}(\Pi))$ then the vector 
$$
w_r := \mathcal{G}(\omega_{\Pi_f})^{-r}w_0
$$
is fixed by all such $\sigma_r$. Hence the ${\mathbb Q}(\Pi)$-span of the 
$G_n({\mathbb A}_f)$-orbit of $w_r$ is the rational structure for this new action; we denote this rational 
structure by $W(\Pi_f)_r$. We have 
$$
W(\Pi_f)_r = \mathcal{G}(\omega_{\Pi_f})^{-r} \, W(\Pi_f)_0.
$$
The comparison map $\mathcal{F}_{\Pi_f,\epsilon,[\Pi]_{\infty}} : W(\Pi_f) \to H(\Pi)$ is the same map as before.  
(For brevity, we abbreviate $H^{b_n}(\mathfrak{g}_{\infty},K_{\infty}^0, V_{\Pi}\otimes M_{\mu}^{\vee})(\epsilon)$
as $H(\Pi)$.) The normalization of this map is different, and we define a period $p^{\epsilon}_r(\Pi)$ by the 
requirement that the normalized map 
$$
\mathcal{F}^r = p^{\epsilon}_r(\Pi)^{-1} \mathcal{F}
$$
maps the rational structure $W(\Pi)_r$ into the rational structure $H(\Pi)_0$; the latter being as before. 
As in \cite[Definition/Proposition 3.3]{raghuram-shahidi-imrn} one can give this definition in an ${\rm Aut}({\mathbb C})$-equivariant manner. For these periods $p_r^{\epsilon}(\Pi)$, the main theorem of \cite{raghuram-shahidi-imrn} 
looks like: 
$$
\sigma\left(\frac
{p_r^{\epsilon \cdot \epsilon_{\xi}}(\Pi_f\otimes\xi_f)}
{\mathcal{G}(\xi_f)^{n(n-1)/2-nr}\,p_r^{\epsilon}(\Pi_f) }\right) 
=
\left(\frac
{p_r^{\epsilon^{\sigma}\cdot\epsilon_{\xi^{\sigma}}}(\Pi_f^{\sigma}\otimes\xi_f^{\sigma})}
{\mathcal{G}(\xi_f^{\sigma})^{n(n-1)/2-nr}\,p^{\epsilon^{\sigma}}(\Pi_f^{\sigma})}\right)
$$
for any algebraic Hecke character $\xi$ of ${\mathbb Q}$. 

It is tempting to stop at this moment and observe that if $n$ is even, and we put $r = (n-2)/2$, then 
the periods $p^{\epsilon}_{(n-2)/2}(\Pi)$ have the same behaviour, upon twisting by Dirichlet characters, 
as the motivic periods of Deligne; the latter being known by Blasius \cite{blasius2} or Panchishkin 
\cite{panchishkin}. However, it is not clear at the moment if $p^{\epsilon}_{(n-2)/2}(\Pi)$ indeed captures
the possibly transcendental part of a critical value of the standard $L$-function of $\Pi$. 
We return to this theme about twisting in Section~\ref{sec:twisted}. We also formulate Conjecture~\ref{con:period-relations} describing a relation between 
the periods of the type $p^{\epsilon}_0(\Pi)$ and Deligne's motivic periods. 

Using the action $\sigma_r$, and the corresponding periods $p_r^{\epsilon}(\Pi)$, the main identity 
of Theorem~\ref{thm:main-identity} looks like: 
$$
\frac{L_f(1/2, \Pi \times \Sigma)}
{p_r^{\epsilon}(\Pi)\, p_r^{\eta}(\Sigma)\, p_{\infty}(\mu,\lambda)} \ = \  
\frac{\prod_{v \in S_{\Sigma}} L_v(1/2, \Pi_v \times \Sigma_v)}
{{\rm vol}(\Sigma)\, \prod_{v \notin S_{\Sigma}} c_{\Pi_v}}\, 
\langle \vartheta^r_{\Sigma,\eta}, \vartheta^r_{\Pi,\epsilon}\rangle_{\mathcal{C}(R_f)},
$$
with the classes defined as 
$$
\vartheta^r_{\Pi,\epsilon} = \mathcal{F}^r_{\Pi_f,\epsilon,[\Pi_{\infty}]}(w_{\Pi_f}),\ \ {\rm and}\ \ 
\vartheta^r_{\Sigma,\eta} = \mathcal{F}^r_{\Sigma_f,\eta,[\Sigma_{\infty}]}(w_{\Sigma_f}),
$$
where the global vectors $w_{\Pi_f}$ and $w_{\Sigma_f}$ are the same vectors as in \ref{subsec:whittaker}.
The action of $\sigma$ on these classes can be read off using (\ref{eqn:r-action}) and 
Proposition~\ref{prop:galois-classes}. In terms of the periods for $\sigma_r$, 
Theorem~\ref{thm:rankin-selberg} on the central critical value now looks like:
{\small
$$
\sigma\left(\frac{L_f(1/2,\Pi \times \Sigma)}
{p_r^{\epsilon}(\Pi) \, p_r^{\eta}(\Sigma) \,
\mathcal{G}(\omega_{\Pi_f})^r \, \mathcal{G}(\omega_{\Sigma_f})^{r+1}\, 
p_{\infty}(\mu,\lambda)}\right) = 
\frac{L_f(1/2, \Pi^{\sigma} \times \Sigma^{\sigma})}
{p_r^{\epsilon}(\Pi^{\sigma}) \, p_r^{\eta}(\Sigma^{\sigma}) \, 
\mathcal{G}(\omega_{\Pi_f^{\sigma}})^r \, \mathcal{G}(\omega_{\Sigma_f^{\sigma}})^{r+1} \, 
p_{\infty}(\mu,\lambda)}.
$$}

The moral of this section is an obvious one that one might have some freedom in defining periods, and 
proving relations amongst such periods, however, the $L$-functions are far more rigid; in the sense that the relations
between $L$-values are more rigid than period relations.

\section{Twisted $L$-functions}
\label{sec:twisted}

Given a cuspidal representation $\Pi$ of $G_n({\mathbb A})$, and a Dirichlet character $\chi$, 
it is often of interest to know the behaviour of the critical values of $L(s, \Pi\otimes\chi)$ when 
we fix the critical point and the representation $\Pi$ and let the character $\chi$ vary. One application
of such a question is toward $p$-adic $L$-functions. 

\subsection{A conjecture of Blasius and Panchishkin}

We now briefly review a conjecture independently due to Blasius \cite[Conjecture L.9.8]{blasius2} and 
Panchishkin \cite[Conjecture 2.3]{panchishkin} about twisted $L$-values. 
Let $\Pi$ be a regular algebraic cuspidal representation of ${\rm GL}_n({\mathbb A})$.
We define $\eta(\Pi) \in \{\pm 1\}$ by 
$$
\eta(\Pi) = {\rm Tr}(\tau(\Pi_{\infty})(j)),
$$
where $\tau(\Pi_{\infty})$ is the Langlands parameter of the representation $\Pi_{\infty}$, which, we recall, is an 
$n$-dimensional semisimple representation of the Weil group $W_{\mathbb R} = {\mathbb C}^* \cup j{\mathbb C}^*$ 
of ${\mathbb R}$. Define $d^{\pm}(\Pi) \in {\mathbb Z}$ by 
$$
d^{\pm}(\Pi) = \left\{\begin{array}{ll}
n/2 & \mbox{if $n$ is even,} \\
(n \pm \eta(\Pi))/2 & \mbox{if $n$ is odd.}\end{array}\right.
$$

\begin{con}
\label{con:blasius}
Let $\Pi$ be a regular algebraic cuspidal representation of ${\rm GL}_n({\mathbb A}_{\mathbb Q})$.
Let $\chi$ be an even Dirichlet character, which is thought of as a Hecke character. 
Note that both $L_f(s, \Pi)$ and $L_f(s, \Pi\otimes\chi)$ have 
the same set of critical points. Let $m$ be such a common critical point; $m \in (n-1)/2 + {\mathbb Z}$.
We have 
$$
L_f(m, \Pi \otimes \chi) \sim_{{\mathbb Q}(\Pi,\chi)} 
\mathcal{G}(\chi_f)^{d^{\pm}(\Pi \otimes |\!|\ |\!|^{(n-1)/2})}
L_f(m, \Pi),
$$
where ${\mathbb Q}(\Pi,\chi)$ denotes the number field generated by the values of the Dirichlet character $\chi$ 
and ${\mathbb Q}(\Pi)$; the sign $\pm$ is $(-1)^{m-(n-1)/2}$;  and 
$d^{\pm}(\Pi\otimes |\!|\ |\!|^{(n-1)/2})$ is 
as defined above. 
\end{con}

We note that if $n$ is even, then the conjecture simplifies to 
$$
L_f(m, \Pi \otimes \chi) \sim_{{\mathbb Q}(\Pi,\chi)} \mathcal{G}(\chi_f)^{n/2} L_f(m, \Pi),
$$

\subsection{Proof of Theorem~\ref{thm:twisted}}
\begin{proof}
We now prove Theorem~\ref{thm:twisted} about the behaviour of central critical value of Rankin--Selberg 
$L$-functions for ${\rm GL}_n \times {\rm GL}_{n-1}$ upon twisting by even Dirichlet characters. 
We go back to earlier notation. Note that Theorem~\ref{thm:rankin-selberg} implies that
$$
L_f(1/2,\Pi \times \Sigma) \ 
\sim_{{\mathbb Q}(\Pi, \Sigma)}\ 
p^{\epsilon}(\Pi) 
p^{\eta}(\Sigma)
\mathcal{G}(\omega_{\Sigma_f})p_{\infty}(\mu,\lambda)
$$
Let $\xi$ be an even Dirichlet character, then the pair $(\Pi\otimes\xi, \Sigma)$ also satisfy the hypotheses of 
Theorem~\ref{thm:rankin-selberg}, with the same pair of highest weights $(\mu,\lambda)$, 
since $\xi_{\infty}$ is trivial. Hence,
$$
L_f(1/2, (\Pi\otimes\xi) \times \Sigma) \ 
\sim_{{\mathbb Q}(\Pi, \Sigma, \xi)}\ 
p^{\epsilon}(\Pi \otimes \xi) 
p^{\eta}(\Sigma)
\mathcal{G}(\omega_{\Sigma_f})p_{\infty}(\mu,\lambda).
$$
We invoke \cite[Theorem 4.1]{raghuram-shahidi-imrn} as rewritten in (\ref{eqn:variation2}) 
to get 
$$
p^{\epsilon}(\Pi\otimes\xi) \sim_{{\mathbb Q}(\Pi, \xi)}
\mathcal{G}(\xi_f)^{n(n-1)/2}p^{\epsilon}(\Pi).
$$
Putting the above together gives
$$
L_f(1/2, (\Pi\otimes\xi) \times \Sigma) \ 
\sim_{{\mathbb Q}(\Pi, \Sigma, \xi)}\ 
\mathcal{G}(\xi_f)^{n(n-1)/2} L_f(1/2,\Pi \times \Sigma).
$$
\end{proof}

\subsection{Some remarks}

\subsubsection{}
Note that in the proof of Theorem~\ref{thm:twisted}, we could have absorbed the twisting character $\xi$ into $\Sigma$ since 
$$
L_f(s,(\Pi\otimes\xi) \times \Sigma) = L_f(s,\Pi \times (\Sigma\otimes\xi)).
$$
If we started with twisting $\Sigma$ by $\xi$, then we would only get $\mathcal{G}(\xi_f)^{(n-1)(n-2)/2}$ by 
applying (\ref{eqn:variation2}) to $p^{\eta}(\Sigma\otimes\xi)$. 
However, there is also the term involving the Gauss sum of $\omega_{\Sigma}$, and since the 
central character transforms as 
$\omega_{\Sigma\otimes\xi} = \xi^{n-1}\omega_{\Sigma}$, 
from \cite[Lemma 8]{shimura1} we have
$$
\mathcal{G}(\xi_f^{n-1}\omega_{\Sigma_f}) \sim_{{\mathbb Q}(\omega_{\Sigma},\xi)} 
\mathcal{G}(\xi_f)^{n-1}\mathcal{G}(\omega_{\Sigma_f}),
$$
i.e., we get the same net contribution of the Gauss sum.

\subsubsection{The Blasius-Panchishkin conjecture for some cusp forms on ${\rm GL}_6$}
We record that Theorem~\ref{thm:twisted} implies Conjecture~\ref{con:blasius} for certain cuspidal 
automorphic representations 
of ${\rm GL}_6({\mathbb A})$. 

\begin{cor}
Let $\Pi$ (resp., $\Sigma$) be a regular algebraic representation of ${\rm GL}_3({\mathbb A})$ 
(resp., ${\rm GL}_2({\mathbb A})$). Let $\Xi= \Pi \boxtimes \Sigma$ be the automorphic representation of ${\rm GL}_6({\mathbb A})$ which is the Kim-Shahidi transfer 
(for the $L$-homomorphism ${\rm GL}_2 \times {\rm GL}_3 \to {\rm GL}_6$
given by tensor product) of the pair $(\Pi,\Sigma)$. 
Assume that $\Xi$ is regular and cuspidal (it is necessarily algebraic), and that 
$s = 1/2$ is critical for $L_f(s, \Xi)$. Then for any 
any even Dirichlet character $\xi$ we have 
$$
L_f(1/2, \Xi \otimes \xi) \ \sim_{{\mathbb Q}(\Xi,\xi)} \ \mathcal{G}(\xi_f)^3L_f(1/2, \Xi).
$$
\end{cor}

\begin{proof}
The standard $L$-function $L(s, \Xi \otimes \xi)$  is nothing but the Rankin--Selberg $L$-function 
$L(s, \Pi \otimes \xi \times \Sigma)$. We leave the rest of the details to the reader.
\end{proof}

Note that since Kasten and Schmidt \cite{kasten-schmidt} have recently proved Hypothesis~\ref{hypo:nonvanishing}
in the situation of ${\rm GL}_3 \times {\rm GL}_2$, the above {\it corollary is therefore true unconditionally.} 
We also note that using the cuspidality criterion of Ramakrishnan-Wang \cite{ramakrishnan-wang}, and by taking 
$\Pi$ and $\Sigma$ to be regular with parameters unrelated to each other, one can see that the set of cuspidal
representations $\Xi$ to which the corollary applies is a nonempty set! We mention in passing that 
Qingyu Wu \cite{wu} has recently studied the image of this transfer.

\section{Odd symmetric power $L$-functions}
\label{sec:sym-357}

Let $\varphi$ be a primitive holomorphic cusp form on the upper half plane of weight $k$, for 
$\Gamma_0(N)$, with nebentypus character $\omega$. We denote this as $\varphi \in S_k(N,\omega)_{\rm prim}$. 
For any integer $r \geq 1$, consider the  
$r$-th symmetric power $L$-function $L_f(s, {\rm Sym}^r \varphi, \xi)$ attached to $\varphi$, twisted by a 
Dirichlet character $\xi$. In this section we prove Theorem~\ref{thm:sym-357} which 
gives an algebraicity theorem for certain critical values of such $L$-functions when $r$ is an odd integer
$\leq 7$.

\subsection{Some preliminaries}

\subsubsection{Symmetric power $L$-functions}
We will work with the $L$-function $L_f(s, {\rm Sym}^r \varphi, \xi)$ in the automorphic context, toward which we let 
$\pi(\varphi)$ be the cuspidal automorphic representation of ${\rm GL}_2({\mathbb A})$ attached to $\varphi$. 
For any integer $r \geq 1$, Langlands' functoriality predicts the existence of an isobaric automorphic 
representation ${\rm Sym}^r(\pi(\varphi))$ of ${\rm GL}_{r+1}({\mathbb A})$, which 
is known to exist for $r \leq 4$ by the work of Gelbart and Jacquet \cite{gelbart-jacquet},  
Kim and Shahidi \cite{kim-shahidi-annals}, and Kim \cite{kim}. If $L(s, {\rm Sym}^r(\pi(\varphi)))$ 
denotes the standard $L$-function of ${\rm Sym}^r(\pi(\varphi))$, then we have 
$$
L_f(s, {\rm Sym}^r \varphi, \xi) = 
L_f(s - r(k-1)/2, {\rm Sym}^r(\pi(\varphi)) \otimes \xi).
$$

For $r \geq 5$, Langlands' functoriality is not known for the $r$-th symmetric power, however, 
by the work of Kim and Shahidi \cite{kim-shahidi-duke}, 
for $5 \leq r \leq 9$ one does have results about the analytic properties of these $L$-functions. 
Let $S$ be any finite set of places containing archimedean and all ramified places for $\pi(\varphi)$, and 
define the partial $L$-functions $L^S(s, \pi(\varphi), {\rm Sym}^r \otimes\xi)$ as in 
\cite[\S 4]{kim-shahidi-duke}. From \cite[Proposition 4.2]{kim-shahidi-duke} and 
\cite[Proposition 4.5]{kim-shahidi-duke} we have 
\begin{enumerate}
\item $L^S(s, \pi(\varphi), {\rm Sym}^5 \otimes\xi)$ is holomorphic and nonzero in ${\rm Re}(s) \geq 1$;  
\item $L^S(s, \pi(\varphi), {\rm Sym}^7 \otimes\xi)$ is holomorphic and nonzero in ${\rm Re}(s) \geq 1$. 
\end{enumerate}
For $v \in S$, one defines the local factors $L(s, \pi(\varphi)_v, {\rm Sym}^r \otimes \xi_v)$ via 
the local Langlands correspondence. After completing the partial $L$-functions with these local factors, 
one gets that both 
$L_f(s, \pi(\varphi), {\rm Sym}^5 \otimes\xi)$  and $L_f(s, \pi(\varphi), {\rm Sym}^7 \otimes\xi)$ are holomorphic
and nonzero in ${\rm Re}(s) \geq 1$. By abuse of notation, we write 
$$
L_f(s, \pi(\varphi), {\rm Sym}^5 \otimes\xi)  = L_f(s, {\rm Sym}^5(\pi(\varphi))\otimes\xi),
$$ 
and so also for the seventh symmetric power.

\subsubsection{Decomposition of certain Rankin-Selberg $L$-functions}

\begin{lemma}
\label{lem:decomposition}
Let $\sigma$ be a two dimensional representation of some group. Then for $n \geq 2$
$$
{\rm Sym}^n(\sigma) \otimes {\rm Sym}^{n-1}(\sigma) \simeq 
{\rm Sym}^{2n-1}(\sigma) \oplus ({\rm Sym}^{2n-3}(\sigma)\otimes {\rm det}(\sigma))
\oplus \cdots \oplus (\sigma\otimes {\rm det}(\sigma)^{n-1})
$$
\end{lemma}

\begin{proof}
This is Clebsch--Gordon for finite-dimensional representations of ${\rm GL}_2({\mathbb C})$. 
\end{proof}

\begin{cor}
\label{cor:factorize}
Let $\varphi \in S_k(N,\omega)_{\rm prim}$, and let $\pi(\varphi)$ be the associated
cuspidal automorphic representation of ${\rm GL}_2({\mathbb A}_{\mathbb Q})$. Let $n \leq 4$.  
For ${\rm Re}(s) \geq 1$ we have
$$
L_f(s, {\rm Sym}^n(\pi(\varphi)) \times {\rm Sym}^{n-1}(\pi(\varphi))) = 
\prod_{a=1}^n
L_f(s, {\rm Sym}^{2a-1}(\pi(\varphi))\otimes\omega_{\pi(\varphi)}^{n-a}).
$$
Assuming Langlands' functoriality, the above equality holds for all $n \geq 1$.
\end{cor}

\subsubsection{Symmetric power transfers have nontrivial cohomology}

To apply Theorem~\ref{thm:rankin-selberg} to get information about critical values for symmetric power $L$-functions, we need to know that the representation ${\rm Sym}^n(\pi(\varphi))$ has nontrivial cohomology. The following theorem
is essentially due to Labesse and Schwermer \cite{labesse-schwermer}. (See also \cite[\S5]{raghuram-shahidi-aims}.)

\begin{thm}
\label{thm:symmetric-cohomology}
Let $\varphi \in S_k(N,\omega)_{\rm prim}$ with $k \geq 2$. 
Let $n \geq 1$. Assume that
${\rm Sym}^n(\pi(\varphi))$ is a cuspidal representation of
${\rm GL}_{n+1}({\mathbb A})$. Let 
$$
\Pi = {\rm Sym}^n(\pi(\varphi)) \otimes \xi \otimes |\!|\ |\!|^s,
$$
where $\xi$ is a Hecke character such that 
$\xi_{\infty} = {\rm sgn}^{\epsilon}$, with $\epsilon \in \{0,1\}$, 
and $|\!| \ |\!|$ is the ad\`elic norm. We suppose that $s$ and $\epsilon$ 
satisfy: 
\begin{enumerate}
\item If $n$ is even, then let $s \in {\mathbb Z}$ and 
$\epsilon \equiv n(k-1)/2 \pmod{2}$. 
\item If $n$ is odd then, we let $s \in {\mathbb Z}$ if k is even,
and we let $s \in 1/2+{\mathbb Z}$ if k is odd. We impose no condition 
on $\epsilon$. 
\end{enumerate}
Then 
$\Pi \in {\rm Coh}(G_{n+1}, \mu^{\vee})$, 
where $\mu \in X^+_0(T_{n+1})$ is given by 
$$
\mu = \left(
\frac{n(k-2)}{2}+s, \frac{(n-2)(k-2)}{2}+s,\dots, \frac{-n(k-2)}{2}+s
\right) = (k-2)\rho_{n+1}+s, 
$$
with $\rho_{n+1}$ being half the sum of positive roots of ${\rm GL}_{n+1}$.
\end{thm}

\subsection{Proof of Theorem~\ref{thm:sym-357}}

As mentioned above, the proof of Theorem~\ref{thm:sym-357} is obtained by applying Theorem~\ref{thm:rankin-selberg}
when $\Pi$ and $\Sigma$ are two consecutive symmetric power transfers of the representation $\pi(\varphi)$. We have 
already commented that these representations, up to some minor twisting, are cohomological. We need to check that
the other hypotheses of Theorem~\ref{thm:rankin-selberg}, which concern the highest weights $\mu$ and $\lambda$,
also hold for these choices. In the following proposition we record the various choices to be made, which depend on the parities of $n$ and $k$. We also record the critical set for the Rankin--Selberg $L$-function at hand, and note that $s=1/2$ is critical in all the cases we consider. Lastly, we specify the signs 
$\epsilon$ and $\eta$ given by the recipe in Theorem~\ref{thm:main-identity} for the specific choice of representations
in each case. 

\begin{prop}
\label{prop:twisting}
Let $\varphi \in S_k(N,\omega)$ be a primitive cusp form, and $\pi(\varphi)$ the associated cuspidal automorphic 
representation of ${\rm GL}_2({\mathbb A})$. Let $\theta$ be any quadratic odd Dirichlet character. 
Let $\xi$ be any Dirichlet character. (We think of $\theta$ and $\xi$ as Hecke characters.)
\begin{enumerate}
\item $k$-even $\geq 4$, and $n$-even.
  \begin{itemize}
  \item $\Pi = {\rm Sym}^n(\pi(\varphi)) \otimes \theta^{n/2}$; 
        $\mu = (k-2)\rho_{n+1}$; $\epsilon = (-1)^{n(n+1)/2}$.
  \item $\Sigma = {\rm Sym}^{n-1}(\pi(\varphi))\otimes \xi \otimes |\!|\ |\!| $; 
        $\lambda = (k-2)\rho_n + 1$; $\eta = -\epsilon$.
  \item Critical set for $L_f(s,\Pi \times \Sigma) = 
        \{\frac{1-k}{2}, \frac{3-k}{2},\dots,\frac{1}{2},\dots,\frac{k-3}{2}\}$.
  \end{itemize}

\item $k$-even $\geq 4$ and $n$-odd.
  \begin{itemize}
  \item $\Pi = {\rm Sym}^n(\pi(\varphi))\otimes \xi \otimes |\!|\ |\!|$; 
        $\mu = (k-2)\rho_{n+1}+1$; $\epsilon = \eta$. 
  \item $\Sigma = {\rm Sym}^{n-1}(\pi(\varphi))\otimes \theta^{(n-1)/2}$; 
        $\lambda = (k-2)\rho_n$; $\eta = (-1)^{n(n-1)/2}$.
  \item Critical set for $L_f(s,\Pi \times \Sigma) = 
        \{\frac{1-k}{2}, \frac{3-k}{2},\dots,\frac{1}{2},\dots,\frac{k-3}{2}\}$.
  \end{itemize}

\item $k$-odd $\geq 3$ and $n$-even.
  \begin{itemize}
  \item $\Pi = {\rm Sym}^n(\pi(\varphi))$; 
        $\mu = (k-2)\rho_{n+1}$; $\epsilon = (-1)^{n(n+1)/2}$. 
  \item $\Sigma = {\rm Sym}^{n-1}(\pi(\varphi))\otimes \xi \otimes |\!|\ |\!|^{1/2}$; 
        $\lambda = (k-2)\rho_n + 1/2$; $\eta = -\epsilon$.
  \item Critical set for $L_f(s,\Pi \times \Sigma) = 
        \{\frac{2-k}{2}, \frac{4-k}{2},\dots,\frac{1}{2},\dots,\frac{k-2}{2}\}$.
  \end{itemize}

\item $k$-odd $\geq 3$ and $n$-odd.
  \begin{itemize}
  \item $\Pi = {\rm Sym}^n(\pi(\varphi))\otimes \xi \otimes |\!|\ |\!|^{1/2}$; 
        $\mu = (k-2)\rho_{n+1} + 1/2$; $\epsilon = \eta$. 
  \item $\Sigma = {\rm Sym}^{n-1}(\pi(\varphi))$; 
        $\lambda = (k-2)\rho_n$; $\eta = (-1)^{n(n-1)/2}$.
  \item Critical set for $L_f(s,\Pi \times \Sigma) = 
        \{\frac{2-k}{2}, \frac{4-k}{2},\dots,\frac{1}{2},\dots,\frac{k-2}{2}\}$.
  \end{itemize}
\end{enumerate}
\end{prop}

We add some comments to illuminate the various twistings and the assumptions on the weight $k$ in the above proposition. 
\begin{enumerate}

\item Twisting by $\xi$. To apply Corollary~\ref{cor:factorize} to get critical values of a certain odd symmetric 
power, one needs to know the critical values of smaller odd symmetric power $L$-functions {\it twisted} by certain 
characters. 

\item Twisting by a power of $\theta$. The presence of this odd Dirichlet character is dictated by the vagaries of 
Theorem~\ref{thm:symmetric-cohomology} in the case when both $k$ and $n$ are even. 

\item Twisting by $|\!|\ |\!|$ when $k$ is even. This is an artifice introduced so that we are really working 
with the critical point $s=3/2$ where all the L-functions at hand are nonvanishing. We need nonvanishing because 
to apply Corollary~\ref{cor:factorize} we need to invert all but one of the factors on the right hand side. We could
avoid this twist if we had a theorem about simultaneous nonvanishing of twisted $L$-functions at $s=1/2$. As of now, 
the best available theorem along these lines seems to be due to    
Chinta-Friedberg-Hoffstein \cite{chinta-friedberg-hoffstein}, but this is not able to handle the point $s = 1/2$.

\item Twisting by $|\!|\ |\!|^{1/2}$ when $k$ is odd. This is simply to ensure that we are working with an 
algebraic representation. This twist automatically takes care that we are dealing with $L$-functions at $s=1$ 
where they are nonvanishing \cite{jacquet-shalika-inv}. 

\item If $k$ is even (resp., odd) then we take $k \geq 4$ (resp., $k \geq 3$) so that the condition 
$\mu^{\vee} \succ \lambda $ is satisfied. In particular, that we do not say anything about the critical values of odd symmetric power 
$L$-functions of elliptic curves. We note that for $k=1$, none of the symmetric power
$L$-function have critical points! (See \cite{raghuram-shahidi-aims}.)
\end{enumerate}

\begin{proof}[Proof of Proposition~\ref{prop:twisting}]
In each case, one has $\Pi \in {\rm Coh}(G_{n+1}, \mu^{\vee})$ and 
$\Sigma \in {\rm Coh}(G_n,\lambda^{\vee})$ by Theorem~\ref{thm:symmetric-cohomology}. The signs $\epsilon$ and 
$\eta$ are given by Theorem~\ref{thm:main-identity}. The list of critical points is an easy exercise involving 
the $L$-factors at infinity: one can write down the Langlands parameters of the representations $\Pi_{\infty}$
and $\Sigma_{\infty}$, and then write down $L(s , \Pi_{\infty} \times \Sigma_{\infty})$. Now do the same with 
the dual representations, and the list follows in every case from the definition of a critical point. It should be kept in mind that $\Pi \times \Sigma$ is, via functoriality, a representation of ${\rm GL}_{n(n+1)}$, and 
$n(n+1)$ is even; the so-called motivic normalization dictates that one looks at critical points in 
$(n(n+1)-1)/2+{\mathbb Z} = 1/2 + {\mathbb Z}$. We omit the routine details. 
\end{proof}

For the proof of Theorem~\ref{thm:sym-357}
we start with the case when $k$ is even. Successively apply Theorem~\ref{thm:rankin-selberg} for the 
pairs of representations $({\rm Sym}^r(\pi(\varphi)), {\rm Sym}^{r-1}(\pi(\varphi)))$ for $r=1,2,3,4$, 
where the representations are taken with appropriate twisting characters as prescribed by Proposition~\ref{prop:twisting}. The proof repeatedly uses the period relations in 
(\ref{eqn:variation2}), and the fact (\cite[Lemma 8]{shimura1})
that for two Hecke characters $\chi_1$ and $\chi_2$ and $\sigma \in {\rm Aut}({\mathbb C})$ one has
$\sigma(\mathcal{G}(\chi_1\chi_2)/\mathcal{G}(\chi_1)\mathcal{G}(\chi_2)) = 
\mathcal{G}(\chi_1^{\sigma}\chi_2^{\sigma})/\mathcal{G}(\chi_1^{\sigma})\mathcal{G}(\chi_2^{\sigma}).$
We have the following:

\begin{equation}
L_f\left(\frac{3}{2}, \pi(\varphi) \otimes \xi\right) 
\sim
p^{-\xi(-1)}(\pi(\varphi)) \, \mathcal{G}(\xi) \, p_{\infty}((k-2)\rho_2+1).
\end{equation}

{\small
\begin{equation}
L_f\left(\frac{3}{2}, {\rm Sym}^3(\pi(\varphi)) \otimes \xi\right) 
\sim 
p^+({\rm Sym}^2(\pi(\varphi)))\frac{p^{\xi(-1)}(\pi(\varphi))}{p^{-\xi(-1)}(\pi(\varphi))} \, 
\mathcal{G}(\xi)^2 \, 
\frac{p_{\infty}((k-2)\rho_3, (k-2)\rho_2+1)}{p_{\infty}((k-2)\rho_2+1)}.
\end{equation}
}

{\small
\begin{equation}
L_f\left(\frac{3}{2}, {\rm Sym}^5(\pi(\varphi)) \otimes \xi\right) 
\sim 
\frac{p^{-\xi(-1)}({\rm Sym}^3(\pi(\varphi)))}{p^{\xi(-1)}(\pi(\varphi))} \, 
\frac{\mathcal{G}(\xi)^3}{\mathcal{G}(\omega)} \, 
\frac{p_{\infty}((k-2)\rho_4+1, (k-2)\rho_3)}{p_{\infty}((k-2)\rho_3, (k-2)\rho_2+1)}.
\end{equation}
}

{\Small
\begin{equation}
\label{eqn:p-phi-n}
L_f\left(\frac{3}{2}, {\rm Sym}^7(\pi(\varphi)) \otimes \xi\right) 
\sim 
\frac{p^+({\rm Sym}^4(\pi(\varphi)))}{p^+({\rm Sym}^2(\pi(\varphi)))}
\frac{p^{\xi(-1)}({\rm Sym}^3(\pi(\varphi)))}{p^{-\xi(-1)}({\rm Sym}^3(\pi(\varphi)))} \, 
\frac{\mathcal{G}(\xi)^4}{\mathcal{G}(\omega)^3}\, 
\frac{p_{\infty}((k-2)\rho_5, (k-2)\rho_4+1)}{p_{\infty}((k-2)\rho_4+1, (k-2)\rho_3)}.
\end{equation}
}

We omit the proof as it is an extended exercise in book-keeping. Similarly, when the weight 
$k$ is odd, we get the following:

\begin{equation}
L_f\left(1, \pi(\varphi) \otimes \xi\right) 
\sim
p^{\xi(-1)}(\pi(\varphi) \otimes |\!|\ |\!|^{1/2}) \, \mathcal{G}(\xi) \, p_{\infty}((k-2)\rho_2 + 1/2).
\end{equation}

{\Small
\begin{equation}
L_f\left(1, {\rm Sym}^3(\pi(\varphi)) \otimes \xi\right) 
\sim 
p^-({\rm Sym}^2(\pi(\varphi)))
\frac{p^{\xi(-1)}(\pi(\varphi)\otimes |\!|\ |\!|^{1/2})}{p^{-\xi(-1)}(\pi(\varphi)\otimes |\!|\ |\!|^{1/2})} \, 
\mathcal{G}(\xi)^2 \, 
\frac{p_{\infty}((k-2)\rho_3, (k-2)\rho_2+1/2)}{p_{\infty}((k-2)\rho_2+1/2)}.
\end{equation}
}

{\small
\begin{equation}
L_f\left(1, {\rm Sym}^5(\pi(\varphi)) \otimes \xi\right) 
\sim 
\frac{p^{-\xi(-1)}({\rm Sym}^3(\pi(\varphi))\otimes |\!|\ |\!|^{1/2})}
{p^{-\xi(-1)}(\pi(\varphi)\otimes |\!|\ |\!|^{1/2})} \,
\frac{\mathcal{G}(\xi)^3}{\mathcal{G}(\omega)} \, 
\frac{p_{\infty}((k-2)\rho_4+1/2, (k-2)\rho_3)}{p_{\infty}((k-2)\rho_3, (k-2)\rho_2+1/2)}.
\end{equation}
}

{\SMALL
\begin{equation}
L_f\left(1, {\rm Sym}^7(\pi(\varphi)) \otimes \xi\right) 
\sim 
\frac{p^+({\rm Sym}^4(\pi(\varphi)))}{p^-({\rm Sym}^2(\pi(\varphi)))}
\frac{p^{-\xi(-1)}({\rm Sym}^3(\pi(\varphi))\otimes |\!|\ |\!|^{1/2})}
     {p^{\xi(-1)}({\rm Sym}^3(\pi(\varphi))\otimes |\!|\ |\!|^{1/2})} \, 
\frac{\mathcal{G}(\xi)^4}{\mathcal{G}(\omega)^3}\, 
\frac{p_{\infty}((k-2)\rho_5, (k-2)\rho_4+1/2)}{p_{\infty}((k-2)\rho_4+1/2, (k-2)\rho_3)}.
\end{equation}
}
In all the above equations, by $\sim$ we mean up to algebraic quantities in an appropriate rationality field, namely 
the number field ${\mathbb Q}(\varphi, \xi)$. More generally, one can say that the quotient of the two sides is 
equivariant under ${\rm Aut}({\mathbb C})$.

We note that the complex number $p^{\epsilon}(\varphi,2n-1)$ in the statement of Theorem~\ref{thm:sym-357}
is a combination of periods attached to various symmetric power representations. 
For example, from (\ref{eqn:p-phi-n}), one has
$$
p^{\epsilon}(\varphi, 7) = 
\frac{p^+({\rm Sym}^4(\pi(\varphi)))}{p^+({\rm Sym}^2(\pi(\varphi)))}
\frac{p^{\epsilon}({\rm Sym}^3(\pi(\varphi)))}{p^{-\epsilon}({\rm Sym}^3(\pi(\varphi)))}
\mathcal{G}(\omega)^{-3}
$$
when the weight of $\varphi$ is even. By induction on $n$, it is possible to write down an expression for 
$p^{\epsilon}(\varphi,2n-1)$. Similarly, one can write down an expression for $p(m,k)$ in terms of $p_{\infty}(\mu,\lambda)$ for various weights $\mu$ and $\lambda$. 
We omit the tedious details.

\subsection{Twisted symmetric power $L$-functions}

\subsubsection{A special case of \cite[Conjecture 7.1]{raghuram-shahidi-aims}}

In an earlier paper with Shahidi \cite{raghuram-shahidi-aims}, we had formulated a conjecture 
about the behaviour of the special values of symmetric power $L$-functions upon twisting by Dirichlet 
characters. See \cite[Conjecture 7.1]{raghuram-shahidi-aims}. We note that Theorem~\ref{thm:sym-357} 
implies this conjecture for certain odd symmetric power $L$-functions. 

\begin{cor}
Let $\varphi$ and the critical point $m$ be as in Theorem~\ref{thm:sym-357}. Let $\xi$ be an even 
Dirichlet character which we identify with the corresponding Hecke character. For $n \leq 4$ we have
$$
L_f(m, {\rm Sym}^{2n-1}\varphi, \xi) \ \sim \ \mathcal{G}(\xi_f)^n L_f(m,  {\rm Sym}^{2n-1}\varphi),
$$
where, by $\sim$, we mean up to an element of the number field ${\mathbb Q}(\varphi, \xi)$. Moreover, 
the quotient $L_f(m, {\rm Sym}^{2n-1}\varphi, \xi)/(\mathcal{G}(\xi_f)^n L_f(m,  {\rm Sym}^{2n-1}\varphi))$
is ${\rm Aut}({\mathbb C})$-equivariant. 
\end{cor}

\subsubsection{Conjecture~\ref{con:blasius} plus Langlands' functoriality implies 
\cite[Conjecture 7.1]{raghuram-shahidi-aims}}

Our conjecture on twisted symmetric power $L$-values follows from 
the more general conjecture of Blasius and Panchishkin. We note that  
the heuristics on the basis of which 
we formulated \cite[Conjecture 7.1]{raghuram-shahidi-aims} are entirely disjoint from the motivic calculations
of Blasius and Panchishkin which is the basis of Conjecture~\ref{con:blasius}.
In this subsection, we briefly sketch a proof 
of how Conjecture~\ref{con:blasius} plus Langlands' functoriality for the $L$-homomorphism 
${\rm Sym}^n : {\rm GL}_2({\mathbb C}) \to {\rm GL}_{n+1}({\mathbb C})$ implies 
\cite[Conjecture 7.1]{raghuram-shahidi-aims}.

\begin{prop}
Let $\varphi \in S_k(N,\omega)_{\rm prim}$. Let $n \geq 1$ be any integer, $\chi$ an even Dirichlet
character (identified with a Hecke character), and $m$ a critical integer for $L_f(s, {\rm Sym}^n \varphi, \chi)$. Then, assuming Langlands 
functoriality in as much as assuming that ${\rm Sym}^n(\pi(\varphi))$ exists as an automorphic representation 
of ${\rm GL}_{n+1}({\mathbb A})$, Conjecture~\ref{con:blasius} implies 
$$
L_f(m, {\rm Sym}^n\varphi, \chi) \sim 
\mathcal{G}(\chi_f)^{\lceil (n+1)/2 \rceil}
L_f(m, {\rm Sym}^n\varphi),
$$ 
unless $n$ is even and $m$ is odd (to the left of center of symmetry), 
in which case we have
$$
L_f(m, {\rm Sym}^n\varphi, \chi) \sim \mathcal{G}(\chi_f)^{n/2}
L_f(m, {\rm Sym}^n\varphi),
$$  
where $\sim$ means up to an element of 
${\mathbb Q}(\varphi, \chi)$.
\end{prop}

\begin{proof}
We assume that ${\rm Sym}^n(\pi(\varphi))$ is cuspidal, because, if not, then $\varphi$ is either dihedral or
a form of weight $1$. (This follows from Kim--Shahidi \cite{kim-shahidi-duke} and 
Ramakrishnan \cite{ramakrishnan-symm}.) If $\varphi$ is dihedral, then we have verified the conclusion; indeed 
this was one of the heuristics for \cite[Conjecture 7.1]{raghuram-shahidi-imrn}. If 
it has weight $1$, then none of the symmetric power $L$-functions have critical integers 
(\cite[Remark 3.8]{raghuram-shahidi-aims}) and so the conclusion
is vacuously true! 

Let $\Pi = {\rm Sym}^n(\pi(\varphi))$. The restriction to ${\mathbb C}^*$ of the Langlands parameter
of $\Pi_{\infty}$ is given by 
$$
z \mapsto \bigoplus_{i=0}^n z^{(n-2i)(k-1)/2} \, \bar{z}^{(2i-n)(k-1)/2}.
$$
Note that $\Pi$ is algebraic if and only if $(n-2i)(k-1)/2 + n/2$ is an integer, and this is so if and only 
if $nk$ is even. If both $n$ and $k$ are odd then $\Pi \otimes |\!|\ |\!|^{1/2}$ is algebraic. 

We start with the case when $n$ is even. Applying Conjecture~\ref{con:blasius} we have 
\begin{eqnarray*}
L_f(m, {\rm Sym}^n\varphi, \chi) 
& = & 
L_f(m - n(k-1)/2, {\rm Sym}^n(\pi(\varphi))\otimes \chi), \\
& \sim & \mathcal{G}(\chi_f)^{\frac{n+1\pm \eta(\Pi \otimes |\!|\ |\!|^{n/2})}{2}}L_f(m - n(k-1)/2, {\rm Sym}^n(\pi(\varphi))) \\
& \sim & \mathcal{G}(\chi_f)^{\frac{n+1\pm \eta(\Pi \otimes |\!|\ |\!|^{n/2})}{2}}L_f(m, {\rm Sym}^n \varphi ), 
\end{eqnarray*}
where $\pm = m-n(k-1)/2-n/2 = m-nk/2$. Now we compute $\eta(\Pi \otimes |\!|\ |\!|^{n/2})$ toward which 
one can check that the Langlands parameter of $\Pi_{\infty}$ is given by 
$$
{\rm Sym}^n(I(\chi_{k-1})) = \epsilon^{n(k-1)/2} \oplus
\bigoplus_{a=1}^{n/2} I(\chi_{(2a(k-1))}),
$$
where for any integer $b$, $I(\chi_b)$ denotes the induction to $W_{\mathbb R}$ of the character 
$z \mapsto (z/|z|)^b$ of ${\mathbb C}^*$, and $\epsilon$ is the sign character of ${\mathbb R}^*$ 
which is thought of as a character of $W_{\mathbb R}$ via the 
isomorphism $W_{\mathbb R}^{\rm ab} \simeq {\mathbb R}^*$. It is easy to see that on all 
the two dimensional summands the element $j \in W_{\mathbb R}$ has trace equal to $0$, and 
$\epsilon$ maps $j$ to $-1$. We get 
$$
\eta(\Pi \otimes |\!|\ |\!|^{n/2}) = (-1)^{nk/2}.
$$
From this we get 
$$
L_f(m, {\rm Sym}^n\varphi, \chi) \sim 
\mathcal{G}(\chi_f)^{\frac{n+1\pm (-1)^{nk/2}}{2}}L_f(m, {\rm Sym}^n \varphi).
$$
We contend that from here on it is easy to see that the conclusion follows. (It might help the reader to further subdivide into the cases depending on when $k$ is even or odd.)

If $n$ is odd then the exponent of the Gauss sum that factors out is predicted to be $(n+1)/2$ 
($=d^{\pm}(\Pi)$) in both Conjecture~\ref{con:blasius} and so also in the conclusion of the proposition. 
One detail that needs to be circumvented 
is that $\Pi$ is not algebraic if both $n$ and $k$ are odd; for this case we argue as:
\begin{eqnarray*}
L_f(m, {\rm Sym}^n\varphi, \chi) 
& = & 
L_f(m - n(k-1)/2 -1/2, {\rm Sym}^n(\pi(\varphi)) \otimes |\!|\ |\!|^{1/2} \otimes \chi) \\
& \sim & \mathcal{G}(\chi_f)^{\frac{n+1}{2}}L_f(m - n(k-1)/2 -1/2, {\rm Sym}^n(\pi(\varphi))\otimes |\!|\ |\!|^{1/2}) \\
& = & \mathcal{G}(\chi_f)^{\frac{n+1}{2}}L_f(m, {\rm Sym}^n \varphi).
\end{eqnarray*}
\end{proof}

\subsection{Remarks on compatibility with Deligne's conjecture}
We recall the famous conjecture of Deligne \cite[Conjecture 2.8]{deligne}. Let $M$ be a motive, and assume that  
$s = 0$ is critical for the $L$-function $L(s,M)$. 
Deligne attaches two periods $c^{\pm}(M)$ to $M$ by comparing the Betti and de Rham realizations of $M$. He predicts that $L(0,M)/c^+(M)$ is in a suitable number field 
${\mathbb Q}(M)$, and more generally, the ratio is ${\rm Aut}({\mathbb C})$-equivariant. 
One expects that our theorems above are compatible with Deligne's conjecture. 
This expectation is formalized in the following conjecture. 

\begin{con}[Period relations]
\label{con:period-relations}
Let $\Pi$ and $\Sigma$ be regular algebraic cuspidal automorphic representations of ${\rm GL}_n({\mathbb A})$
and ${\rm GL}_{n-1}({\mathbb A})$, respectively. Let $\mu,\lambda$ be the associated highest weights, and 
$\epsilon, \eta$ the associated signs as in Theorem~\ref{thm:rankin-selberg}.
We let $M(\Pi)$ and $M(\Sigma)$ be the conjectural motives attached to $\Pi$ and $\Sigma$  
(by Clozel \cite[Conjecture 4.5]{clozel}). Let $M = M(\Pi) \otimes M(\Sigma)$. Let $c^{\pm}(M)$ be Deligne's 
periods attached to $M$, and $d^{\pm}(M)$ be the integers as in Deligne \cite[\S 1.7]{deligne}.
We expect 
$$
p^{\epsilon}(\Pi) p^{\eta}(\Sigma)\mathcal{G}(\omega_{\Sigma_f})p_{\infty}(\mu,\lambda) 
\sim 
(2\pi i)^{d^+(M)n(n-1)/2}c^{(-1)^{n(n-1)/2}}(M),
$$
where, by $\sim$, we mean up to an element of the number field ${\mathbb Q}(\Pi,\Sigma)$.
\end{con}

The heuristic for the above conjecture is the following `formal' calculation based on 
Theorem~\ref{thm:rankin-selberg}, Langlands' functoriality, 
the correspondence between automorphic representations and motives as in Clozel \cite[Conjecture 4.5]{clozel}, 
and Deligne \cite[Conjecture 2.8]{deligne}: 

\begin{eqnarray*}
p^{\epsilon}(\Pi) p^{\eta}(\Sigma)\mathcal{G}(\omega_{\Sigma})p_{\infty}(\mu,\lambda) 
& \sim &
L(1/2, \Pi \times \Sigma), \ \ \mbox{(by Theorem~\ref{thm:rankin-selberg})}\\
& = & L(1/2, \Pi \boxtimes \Sigma), \ \ \mbox{(by Langlands' functoriality)} \\
& = & L(n(n-1)/2, M(\Pi \boxtimes \Sigma)),\ \ \mbox{(motivic normalization)} \\
& = & L(n(n-1)/2, M), \ \ \mbox{(by definition of $M$)} \\
& = & L(0, M(n(n-1)/2)), \ \ \mbox{(see \cite[3.1.2]{deligne})} \\
& \sim & c^+(M(n(n-1)/2)), \ \ \mbox{(by Deligne \cite[Conjecture 2.8]{deligne})} \\
& = & (2\pi i)^{d^+(M)n(n-1)/2}c^{(-1)^{n(n-1)/2}}(M) \ \ \mbox{(see \cite[5.1.8]{deligne})}.
\end{eqnarray*}

It is possible to express $c^{\pm}(M(\Pi)\otimes M(\Sigma))$ in terms of the periods, or perhaps some other 
finer invariants, attached to $M(\Pi)$ and $M(\Sigma)$, as in Blasius \cite{blasius-appendix}
and Yoshida \cite{yoshida}.


\bigskip

{\it Address:}

A. Raghuram

Department of Mathematics

Oklahoma State University

401 Mathematical Sciences

Stillwater, OK 74078, USA. 

{\it E-mail address:}
{\tt araghur@math.okstate.edu}

\end{document}